 \documentclass[final,1p,times]{elsarticle}
\usepackage{amssymb}
\usepackage{amsthm}
\usepackage{amssymb}

\usepackage{amsmath,amssymb}
\usepackage{natbib}
\newtheorem{theorem}{Theorem}

\newtheorem{lemma}{Lemma}
\newtheorem{corollary}{Corollary}
\newtheorem{defn}{Definition}
\newtheorem{remark}{Remark}

\newtheorem{result}{Result}
\newtheorem{claim}{Claim}
\graphicspath{ {figures/} }

\journal{}

\begin{document}

\begin{frontmatter}

%% Title, authors and addresses

%% use the tnoteref command within \title for footnotes;
%% use the tnotetext command for theassociated footnote;
%% use the fnref command within \author or \address for footnotes;
%% use the fntext command for theassociated footnote;
%% use the corref command within \author for corresponding author footnotes;
%% use the cortext command for theassociated footnote;
%% use the ead command for the email address,
%% and the form \ead[url] for the home page:
%% \title{Title\tnoteref{label1}}
%% \tnotetext[label1]{}
%% \author{Name\corref{cor1}\fnref{label2}}
%% \ead{email address}
%% \ead[url]{home page}
%% \fntext[label2]{}
%% \cortext[cor1]{}
%% \affiliation{organization={},
%%             addressline={},
%%             city={},
%%             postcode={},
%%             state={},
%%             country={}}
%% \fntext[label3]{}

\title{Arbitrary high-order unconditionally stable methods for reaction-diffusion equations with inhomogeneous boundary condition via Deferred Correction\thanks{Part of this work was supported by the NSERC CREATE program “Génie par la Simulation” as a scholarship to the first author. The second author would like to acknowledge the financial support of the Discovery Grant Program of the Natural Sciences and Engineering Research Council of Canada (NSERC). 
}}

%% use optional labels to link authors explicitly to addresses:
%% \author[label1,label2]{}
%% \affiliation[label1]{organization={},
%%             addressline={},
%%             city={},
%%             postcode={},
%%             state={},
%%             country={}}
%%
%% \affiliation[label2]{organization={},
%%             addressline={},
%%             city={},
%%             postcode={},
%%             state={},
%%             country={}}

\author[label1]{Saint-Cyr E.R. Koyaguerebo-Im\'e}
\ead[label1]{skoyague@univ-bangui.org}
\address[label1]{ Université de Bangui, Faculté des Sciences, Département de Mathématiques et Informatique, BP 1450, Bangui, République Centrafricaine,
              Tel.: +236 72 47 83 77
}
\author[label2]{Yves Bourgault}
\ead[label2]{ybourg@uottawa.ca}

\address[label2]{Department of Mathematics and Statistics, University of Ottawa, STEM Complex,\\ 150 Louis-Pasteur Pvt, Ottawa, ON, Canada, K1N 6N5, Tel.:  +613-562-5800x2103
}
%\author{Saint-Cyr E.R. Koyaguerebo-Im\'e}
%\ead{skoyague@univ-bangui.org}
%\address{ Université de Bangui, Faculté des Sciences, Département de Mathématiques et Informatique, BP 1450, Bangui, République Centrafricaine,
%              Tel.: +236 72 47 83 77
%}
\cortext[cor1]{Corresponding author: skoyague@univ-bangui.org}
%\authorrunning{Saint-Cyr E. R. Koyaguerebo-Im\'e \and\\ Yves Bourgaultt}

\begin{abstract}
In this paper we analyse full discretizations of an initial boundary value problem (IBVP) related to reaction-diffusion equations. To avoid possible order reduction, the IBVP is first transformed into an IBVP  with homogeneous boundary conditions (IBVPHBC) via a lifting of inhomogeneous Dirichlet, Neumann or mixed Dirichlet-Neumann boundary conditions. The IBVPHBC is discretized in time via the deferred correction method for the implicit midpoint rule and leads to a time-stepping scheme of order $2p+2$ of accuracy at the stage $p=0,1,2,\cdots $ of the correction. Each semi-discretized scheme results in a nonlinear elliptic equation  for which the existence of a solution is proven using the  Schaefer fixed point theorem. The elliptic equation corresponding to the stage $p$ of the correction is discretized by the Galerkin finite element method and gives a full discretization of the IBVPHBC. This fully discretized scheme is unconditionally stable with order $2p+2$ of accuracy in time. The order of accuracy in space is equal to the degree of the finite element used when the family of meshes considered is shape-regular while an increment of one order is proven for quasi-uniform family of meshes. Numerical tests with a bistable reaction-diffusion equation having a strong stiffness ratio and a linear reaction-diffusion equation addressing order reduction are performed and demonstrate the unconditional convergence of the method. The orders 2,4,6,8 and 10 of accuracy in time are achieved.

\end{abstract}

%%%Graphical abstract
%\begin{graphicalabstract}
%%\includegraphics{grabs}
%\end{graphicalabstract}
%
%%%Research highlights
%\begin{highlights}
%\item Research highlight 1
%\item Research highlight 2
%\end{highlights}


\end{frontmatter}

%% \linenumbers

%% main text
\section*{Introduction}
\label{sec;introduction}
Let $\Omega $ be a bounded domain in $\mathbb{R}^d$ ($ d=1,2,3$) with smooth boundary $\partial \Omega$ and $T>0$. Consider the following reaction-diffusion system 
\begin{equation}
\label{a1}
\left\lbrace \begin{aligned}
u'-M\Delta u +f(u)&=S \mbox{ in } \Omega \times (0,T),\\
B u&=\varphi \mbox{ on } \partial \Omega \times [0,T],\\
u(.,0)&=u_0 \mbox{ in } \Omega ,
\end{aligned}
\right. 
\end{equation}
where $u:\Omega \times [0,T] \rightarrow \mathbb{R}^J$ is the unknown, for a positive integer $J$, $M$ is an $J\times J$ constant matrix, $f:\mathbb{R}^J \rightarrow \mathbb{R}^J$, $S, \varphi \,:\,\Omega \times (0,T) \rightarrow \mathbb{R}^J$ are given smooth functions, and $B$ is an operator denoting Dirichlet, Neumann or mixed Dirichlet-Neumann boundary conditions. This is a general form of reaction-diffusion equations (see for instance \cite{smoller2012shock}) that model various phenomena in physics, combustion, chemical reactions, population dynamics and biomedical science (cancer modelling and other physiological processes) (see, e.g., \cite{pesin2004some, newell1969finite, smoller2012shock, markowich2007applied, volpert2014elliptic}). 
%For $f(u)=au^3+bu^2+cu+d$ we have the Kolmogorov-Petrovskii-Piskunov (KPP) equations for planar model of advance of advantageous genes ($a=1$, $b=-2$, $c=1$), the reduction to one variable of the  FitzHugh-Nagumo model for the propagation of the depolarisation front through a nerve axon ( $f(u)=au(u-\theta)(u-1)$, $a>0$ and $\theta \in [0,1/2]$), the Newell-Whitehead equations used in population genetics ($a=1$, $b=-c-1$, $0<c<1$), (see \cite{pesin2004some, newell1969finite} for more examples of reaction diffusion equations). 

We suppose that $M$ is positive definite and the function $f$ satisfies the following two monotonicity conditions 
\begin{equation}
\label{a2}
\left( f(x)-f(y),x-y\right) \geq \alpha \vert x-y \vert^q +\tau(y) |x-y|^2, \forall x,y\in\mathbb{R}^J, \mbox{for some } \alpha \geq 0,q\geq 1,
\end{equation}
and
\begin{equation}
\label{a3} (df(x)y)\cdot y \geq -\mu_0 |y|^2, ~~\forall x,y\in\mathbb{R}^J,
\end{equation}
where $\mu_0$ is a nonnegative real, and $\tau$ is an arbitrary continuous real-valued function. In addition, for inhomogeneous Neumann boundary conditions, we suppose that 
\begin{equation}
\label{a3ab}\vert df(x)\vert\leq \mu_1\left( 1+|x|^{q-2}\right)  , \mbox{ for each } x\in\mathbb{R}^J, \mbox{ and }2\leq q<18,
\end{equation}
where $\mu_1$ is a positive real. Conditions (\ref{a2})-(\ref{a3}) guarantee the existence of a solution of problem (\ref{a1}) in $L^2\left( 0,T;H^2(\Omega)\right)$ (see for instance \cite{lions1969quelques, temam1997infinite, evans2010partial}),  and uniqueness and high order regularity can be deduced. The conditions (\ref{a2})-(\ref{a3}) are at least satisfied by any  polynomial of odd degree with positive leading coefficient, and the matrix $M$ is supposed to be constant only for the sake of simplicity. In fact, all our results remain true replacing the operator $M\Delta $ by an elliptic operator $L$:
\begin{equation}
\label{a3bb} Lu=-\sum_{i,j=1}^Ja^{i,j}(x)u_{x_ix_j}+\sum_{i=1}^Jb^j(x)u_{x_i}+c_0(x)u,
\end{equation}
where the coefficients $a^{i,j}$, $b^i$ and $c_0$ are smooth functions, and $a^{i,j}=a^{j,i}$ (see, e.g., \cite[p.314]{evans2010partial} for a definition of elliptic operator).

The numerical analysis of reaction-diffusion equations takes advantage of many results available from the numerical analysis of semi-linear parabolic partial differential equations (PDEs). The method of lines (MOL) is commonly used. By this method the PDE is first discretized in space by finite element or finite difference methods, leading to a system of ordinary differential equations (ODEs). The resulting system of ODEs is then discretized by fully implicit or implicit-explicit (IMEX) time-stepping methods (see for instance  \cite{ akrivis2015stability, akrivis2004linearly, akrivis1998implicit, zlamal1977finite,thomee1984galerkin,hoff1978stability,ruuth1995implicit,kress2002deferred}). In \cite{akrivis2015stability, akrivis2004linearly, akrivis1998implicit}, linear implicit-explicit multistep methods in time together with finite element methods in space are analysed for a class of abstract semi-linear parabolic equations that includes a large class of reaction-diffusion systems. The approaches in \cite{akrivis2015stability, akrivis2004linearly, akrivis1998implicit} are the same. The authors investigate approximate solutions expected to be in a tube around the exact solution. They proceeded by induction and established that if $k$ and $k^{-1}h^{2r}$, $r\geq 2$, are small enough then the global error of the scheme is of order $p$ ($p=1,2,...,5$) in time and $r$ in space. IMEX schemes with finite difference in space and Runge-Kutta of order 1 and 2 in time are also analysed in \cite{koto2008imex,bujanda2004efficient} for a class of reaction-diffusion systems.  Otherwise, in \cite{zlamal1977finite, madzvamuse2014fully,thomee1984galerkin} fully implicit numerical methods for reaction-diffusion equations with restrictive conditions on the nonlinear term are introduced, combining finite elements in space and backward Euler, Crank-Nicolson or fractional-step $\theta$ methods in time. The resulting schemes are unconditionally stable (the time step is independent from the space step) with order 1 or 2 of accuracy in time. The time-stepping method in \cite{kress2002deferred} is constructed via a deferred correction strategy applied to the trapezoidal rule and is of arbitrary high order. However, this method concerns only linear initial value problems (IVP) (resulting eventually from a MOL) satisfying a monotonicity condition and has an issue for the starting procedure. Furthermore, the stability analysis proposed in \cite{kress2002deferred} does not guarantee unconditional stability and/or an optimal a priori error estimate, when a full discretization is considered.

In practice, the space-discretization of time-evolution PDEs leads to a stiff IVP of large dimension (we recall that a stiff problem is a problem extremely hard to solve by standard explicit methods (see, e.g., \cite{spijker1996stiffness}). To avoid overly small time steps, accurate approximate solutions for these IVPs require high order time-stepping methods having good stability properties (A-stable methods  satisfying B-convergence property are of great interest). Backward differentiation formulae (BDF) of order 1 and 2 are commonly used according to their A-stability. However, BDF methods of order 3 and higher lack stability properties (e.g. for systems with complex eigenvalues).  Moreover, Runge-Kutta methods applied to such IVPs require high computational cost and are prone to order reduction (see \cite{sanz1986convergence} or the section 2.12 of reference \cite{kennedy2016diagonally} that reviews papers dealing with order reduction about RK methods in the case of time-evolution PDEs). The paper \cite{alonso2002runge} propose a strategy to avoid order reduction for RK methods applied to linear boundary value problems, but the resulting schemes loose unconditional stability. 

The aim of this paper is to analyse the full discretization of problem (\ref{a1}) by applying the deferred correction (DC) method introduced in \cite{koyaguerebo2021arbitrary} for the time-discretization and a finite element discretization in space. The DC method in \cite{koyaguerebo2021arbitrary} addresses ODEs and consists in a successive perturbation (correction) of the implicit midpoint rule, leading to a $A$-stable and $B$-convergent scheme of order $2p+2$ of accuracy at the stage $p=0,1,2, \cdots$ of the correction. We recall that the order of accuracy of the DC schemes is guaranteed by a deferred correction condition $(DCC)$. 

In our approach, we first  transform problem (\ref{a1}), having $u$ as exact solution, into  the following boundary value problem with homogeneous boundary condition and that has exact solution $\bar{u}=u-\varphi$:
\begin{equation}
\label{aa1}
\left\lbrace \begin{aligned}
\bar{u}'-M\Delta \bar{u} +f(\bar{u}+\varphi)&=\tilde{S}& \mbox{ in } \Omega \times (0,T),\\
B\,\bar{u}&=0& \mbox{ on } \partial \Omega \times [0,T],\\
\bar{u}(0)&=u_0-\varphi(0)& \quad\mbox{ in } \Omega ,
\end{aligned}
\right. 
\end{equation}
where
\begin{equation}
\label{ab1}\tilde{S}(x,t)=S(x,t)-\varphi'(x,t)+M \Delta \varphi(x,t),\,\,\, \forall (x,t)\in \Omega \times [0,T].
\end{equation}
This strategy is already adopted in \cite{sanz1986convergence}, for Dirichlet boundary conditions related to a linear hyperbolic problem, and is useful to avoid order reduction from an artificial treatment of inhomogeneous boundary conditions. According to the fact that the trace operator is surjective \cite[Thm. B54]{ern2013theory}, $\varphi$ can be replaced by any sufficiently smooth function that agree  with the boundary condition of $u$. Therefore, to find an approximate solution of $u$, we are only interested in an approximate solution of $\bar{u}$ and retrieve that of $u$ from $u=\bar{u}+\varphi$. The main ingredient to deduce the analysis of convergence of our full discretizations of problem (\ref{aa1}) from results in \cite{koyaguerebo2021arbitrary} is the MOL. Unfortunately, the MOL is inefficient to guarantee unconditional stability, due to difficulties/impossibility to prove $DCC$ for the fully discretized schemes independently of the space step. To overcome these difficulties, we start our analysis from the time semi-discretization of problem (\ref{aa1}) by the DC method. Each time semi-discretized scheme results in an elliptic boundary value problem which existence and basic regularity properties are proven using the Schaefer fixed point theorem and elliptic regularity results. We prove that each semi-discrete solution from the stage p of the correction converges to the exact solution with order $2p+2$ of accuracy. Similarly to the case of IVP in \cite{koyaguerebo2021arbitrary}, the proof of convergence in time up to order $2p+2$ is guaranteed by a $DCC$ satisfied by the semi-discrete solutions, but, as a trick to simplify the proof, we suppose that the exact solution u of (\ref{a1}) is stationary in a small time interval $[0,(2p+1)k_0]$, where $k_0$ is a maximal time step for the time semi-discrete solutions and satisfies $k_0\mu_0 <2$ ($\mu_0$ is the constant introduced in (\ref{a3})). Various convergence results concerning the time semi-discrete solutions, which are useful for the analysis of convergence of the full discrete solutions, are proven. We consider full discretization of problem (\ref{aa1}) as discretization of the time semi-discrete schemes by Galerkin finite element method. Proof of existence of the fully discrete solutions requires an improvement of the lemma on zeros of a vector field. We prove that each fully discrete solution converges unconditionally in space to the corresponding time semi-discrete solution. The order of accuracy in space is $r+1$ when a finite element of degree r is used. The unconditional convergence in space implies an unconditional convergence both in time and space of the fully discrete solutions to the exact solution. The order of accuracy in time of the fully discrete solution corresponding to the stage p of the correction is $2p+2$. Numerical illustration with the schemes of order 2, 4, 6, 8 and 10 in time, using the bistable reaction-diffusion equation and a linear reaction-diffusion equation addressing order reduction, are given.
%% The outline is not required, but we show an example here.

The paper is organized as follows. We recall some algebraic properties of finite difference operators in section \ref{sec:FE}. In section \ref{sec:SD} we introduce the semi-discretized schemes in time  and prove the existence of a solution. The analysis of convergence and order of accuracy of solutions for the semi-discretized schemes in time is done in section \ref{sec:Semi}. The fully discretized schemes are presented and analysed in section \ref{sec:Conv}, and numerical experiments are carried in section \ref{sec:numerical experiments}. 

\section{Finite difference operators}
\label{sec:FE} In this section we recall main results from finite difference (FD) approximations. Details and proofs for these results can be found in \cite{koyaguerebo2020finite}. For a time step $k>0$, we denote $t_n=nk$ and $t_{n+1/2}=(n+1/2)k$, for each integer $n$. This implies that $t_0=0$. We consider the time steps $k$ such that $0=t_0<t_1<\cdots< t_N=T$ is a partition of $[0,T]$, for a nonnegative integer $N$. The centered, forward and backward difference operators $D$, $D_+$ and $D_-$, respectively, related to $k$ and applied to a function $v$ from $[0,T]$ into a Banach space $X$ (with norm $\|\cdot\|_X$), are defined as follows:
$$Dv(t_{n+1/2})=\frac{v(t_{n+1})-v(t_n)}{k},$$
$$D_+v(t_{n})=\frac{v(t_{n+1})-v(t_n)}{k},$$ 
and 
$$D_-v(t_{n})=\frac{v(t_{n})-v(t_{n-1})}{k}.$$ The average operator is denoted by $E$:
$$E v(t_{n+1/2})=\widehat{v}(t_{n+1})=\frac{v(t_{n+1})+v(t_n)}{2}.$$
The composites of $D_+$ and $D_-$ are defined recursively. They commute, that is $(D_+D_-)v(t_n)=(D_-D_+)v(t_n)=D_-D_+v(t_n)$, and satisfy the identities
\begin{equation}
\label{bb1}
(D_+D_-)^mv(t_n)=k^{-2m} \sum_{i=0}^{2m}(-1)^i {{2m}\choose {i}}v(t_{n+m-i} ),
\end{equation}and
\begin{equation}
\label{bb2}
D_-(D_+D_-)^mv(t_n)=k^{-2m-1}\sum_{i=0}^{2m+1}(-1)^i{{2m+1}\choose {i}}v(t_{n+m-i}),
\end{equation}
for each integer $m\geq 1$ such that $0\leq t_{n-m-1}\leq t_{n+m}\leq T$. If $\left\lbrace v^n \right\rbrace_n$ is a sequence of approximation of $v$ at the discrete points $t_n$, the finite difference operators apply to $\left\lbrace v^n \right\rbrace $ and we define
$$
Dv^{n+1/2}=D_+v^{n}=D_-v^{n+1}=\frac{v^{n+1}-v^n}{k},$$and
$$E v^{n+1/2}=\widehat{v}^{n+1}=\frac{v^{n+1}+v^n}{2}.$$

We have the following three results:
\begin{result}
\label{prop:1}
For nonnegative integers $m_1$ and $m_2$, provided $v\in C^{m_1+m_2}([0,T],X)$ and $m_2\leq n\leq N-m_1$, we have
\begin{equation}
\label{b3}
\left \| D_+^{m_1}D_-^{m_2}v(t_n)\right \| \leq  \max_{t_{n-m_2}\leq t\leq t_{n+m_1}}\left \|\frac{d^{m_1+m_2}v}{dt^{m_1+m_2}} (t)\right \|.
\end{equation}
%\begin{theorem}
%\label{prop:1}
%Let $v$ be $C^{m}([0,T],X)$, $m=1,2,...$, and $0=t_0<t_1<...<t_N=T$, $t_n=nk$, be a partition $[0,T]$. Let $m_1$ and $m_2$ be two positive integers such that $m_1+m_2\leq m$.  Then, for each integer $n$ such that $m_2\leq n\leq N-m_1$, $D_+^{m_1}D_-^{m_2}v(t_n)$ is bounded independently of $n$, and we have the estimate
%\begin{equation}
%\label{b3}
%\left \| D_+^{m_1}D_-^{m_2}v(t_n)\right \| \leq  \max_{t_{n-m2}\leq t\leq t_{n+m_1}}\left \|v^{(m_1+m_2)}(t) \right \|,
%\end{equation} where
%$$v^{(m_1+m_2)}(t)=\frac{d^{m_1+m_2}v}{dt^{m_1+m_2}}(t).$$
%\end{theorem}
\end{result}
\begin{result}[Central finite difference approximations] There exists a sequences $\left\lbrace c_i \right\rbrace _{i\geq 2}$ of real numbers such that, for all $v\in C^{2p+3}\left( [0,T],X\right)$, where $p$ is a positive integer, and $p\leq n\leq N-1-p$, we have
\begin{equation}
\label{b6} 
v'(t_{n+1/2})=
\frac{v(t_{n+1})-v(t_n)}{k}-\sum_{i=1}^pc_{2i+1}k^{2i}D(D_+D_-)^iv(t_{n+1/2}) +O(k^{2p+2}),
\end{equation}
and
\begin{equation}
\label{b7} 
v(t_{n+1/2}) =\frac{v(t_{n+1})+v(t_n)}{2}-\sum_{i=1}^pc_{2i}k^{2i}(D_+D_-)^iE v(t_{n+1/2})+O(k^{2p+2}).
\end{equation} 

 Table \ref{tab:1} gives the ten first coefficients $c_i$.

\begin{table}[h!] 
\caption{Ten first coefficients of central difference approximations (\ref{b6}) and (\ref{b7})} \centering     
\label{tab:1}
 % used for centering table 
\begin{tabular}{cccccccccc}  % centered columns (4 columns)
\hline\noalign{\smallskip} 
    $c_2$ &$c_3$ & $c_4$ &$c_5$& $c_6$&$c_7$&$c_8$&$c_9$&~$c_{10}$&$c_{11}$\\[.75ex]
    $\frac{1}{8}$        &$\frac{1}{24}$         &$-\frac{18}{4!2^5}$      &$-\frac{18}{5!2^5}$     &$\frac{450}{6!2^7}$      &$\frac{450}{7!2^7}$      &$-\frac{22050}{8!2^9}$   &$-\frac{22050}{9!2^9}$  &$\frac{1786050}{10!2^{11}}$     & $\frac{1786050}{11!2^{11}}$\\[.75ex]
\noalign{\smallskip}\hline
\end{tabular} % is used to refer this table in the text 
\end{table} 
\end{result} 

%\begin{theorem}
%\label{theo:1} There exists a sequence $\displaystyle \left\lbrace c_i \right\rbrace_{i\geq 2}$ of real numbers such that for any function $u\in C^{2p+3}\left( [0,T],X\right)$ and a partition $0=t_0<t_1<...<t_N=T$, $t_n=nk$, of $[0,T]$, we have
%\begin{equation}
%\label{b6} 
%u'(t_{n+1/2})=
%\frac{u(t_{n+1})-u(t_n)}{k}-\sum_{i=1}^pc_{2i+1}k^{2i}D(D_+D_-)^iu(t_{n+1/2}) +O(k^{2p+2})
%\end{equation}
%and
%\begin{equation}
%\label{b7} 
%u(t_{n+1/2}) =\frac{u(t_{n+1})+u(t_n)}{2}-\sum_{i=1}^pc_{2i}k^{2i}(D_+D_-)^iE u(t_{n+1/2})+O(k^{2p+2}),
%\end{equation} for $p\leq n\leq N-1-p$. The error constants for the formulae (\ref{b6}) and (\ref{b7}) are, respectively, $c_{2p+3}$ and $c_{2p+2}$. Table \ref{tab:1} gives the first ten coefficients $c_i$.
%
%\begin{table}[h!] 
%\caption{Ten first coefficients of central difference approximations (\ref{b6}) and (\ref{b7})} \centering     
%\label{tab:1}
% % used for centering table 
%\begin{tabular}{llllllllll}  % centered columns (4 columns)
%\hline\noalign{\smallskip} 
%    $c_2$ &~$c_3$ &~~~ $c_4$ &~~~~$c_5$& ~$c_6$&~~$c_7$&~~~~$c_8$&~~~~~$c_9$&~~~$c_{10}$&~~~$c_{11}$\\[.5ex]
%    $\frac{1}{8}$        &$\frac{1}{24}$         &$-\frac{18}{4!2^5}$      &$-\frac{18}{5!2^5}$     &$\frac{450}{6!2^7}$      &$\frac{450}{7!2^7}$      &$-\frac{22050}{8!2^9}$   &$-\frac{22050}{9!2^9}$  &$\frac{1786050}{10!2^{11}}$     & $\frac{1786050}{11!2^{11}}$\\[.5ex]
%\noalign{\smallskip}\hline
%\end{tabular} % is used to refer this table in the text 
%\end{table} 
%\end{theorem}

\begin{result}[Interior central finite difference approximations]
For each positive integer $p$ there exists reals $c_2^p,c_3^p,\cdots,c_{2p+1}^p$ such that, for each $v\in C^{2p+3}\left( [a,b],X\right)$ and a uniform  partition $a=\tau_0<\tau_1<...<\tau_{2p+1}=b$ of the interval $[a,b]$, with $\tau_n=a+nk$, $k=(b-a)/(2p+1)$ and $\tau_{p+1/2}=(a+b)/2$, we have  
\begin{equation}
\label{01} 
u'(\tau_{p+1/2})=
\frac{u(b)-u(a)}{b-a}-\frac{1}{b-a}\sum_{i=1}^pc^p_{2i+1}k^{2i+1}D(D_+D_-)^iu(\tau_{p+1/2}) +O(k^{2p+2}),
\end{equation}
and
\begin{equation}
\label{02} 
u(\tau_{p+1/2}) =\frac{u(b)+u(a)}{2}-\sum_{i=1}^pc^p_{2i}k^{2i}(D_+D_-)^iE u(\tau_{p+1/2})+O(k^{2p+2}).
\end{equation} 
Table \ref{tab:3} gives the coefficients ${c}^p_i$ for $p=1,2,3,4$.

\begin{table}[!ht] 
\caption{Coefficients of the approximations (\ref{01})-(\ref{02}) for $p=1,2,3,4$}
\label{tab:3}
 \centering      % used for centering table 
\begin{tabular}{ccccccccccc}  % centered columns (4 columns)
\hline\noalign{\smallskip} 
   $p$& ${c}^p_2$ &~${c}^p_3$ &~~~ ${c}^p_4$ &~~~~${c}^p_5$& ~${c}^p_6$&~~${c}^p_7$&~~~~${c}^p_8$&~~~~~${c}^p_9$&~~~\\[.75ex]
  1&  ~$\frac{9}{8}$        &~$\frac{9}{8}$       \\[.75ex]
  2&  $\frac{25}{8}$        &$\frac{125}{24}$         &$\frac{125}{128}$      &$\frac{125}{128}$     \\[.75ex]
  3&  $\frac{49}{8}$        &$\frac{343}{24}$         &$\frac{637}{128}$      &$\frac{13377}{1920}$     &$\frac{1029}{1024}$      &$\frac{1029}{1024}$      \\[.75ex]
  4&  $\frac{81}{8}$        &$\frac{243}{8}$         &$\frac{1917}{128}$      &$\frac{17253}{640}$     &$\frac{7173}{1024}$      &$\frac{64557}{7168}$      &$\frac{32733}{32768}$   &$\frac{32733}{32768}$  \\[.75ex]
\noalign{\smallskip}\hline
\end{tabular} % is used to refer this table in the text 
\end{table} 
\end{result}

\section{Semi-discrete schemes in time: existence of a solution} 
\label{sec:SD}
Hereafter we suppose that (\ref{a1}) has a unique solution $u\in C^{2p+4}\left([0,T],H^{r+1}(\Omega) \right) $, for some positive integers $p$ and $r$. We denote by $\left(\cdot ,\cdot\right) $ the inner product in $L^2(\Omega)$ and by $\|\cdot\|$ the corresponding norm. The norm in the Sobolev spaces $H^m(\Omega)$ will be noted $\|\cdot\|_m$, for each nonnegative integer $m$, and we note $\|\cdot\|_{\infty}=\|\cdot\|_{L^\infty(\Omega)}$. We use $h$ and $k$ to denote stepsizes for space and time discretizations, respectively. The letter $C$ will denote any constant independent from $h$ and $k$, and that can be calculated explicitly in term of known quantities. The exact value of $C$ may change.

As in \cite{koyaguerebo2021arbitrary}, we can apply deferred correction method to (\ref{aa1}) and deduce the following schemes:

{For $j=0$, we have the \textit{implicit midpoint rule}}

\begin{equation}
\label{c1} 
\left\lbrace 
\begin{array}{ccc}
\displaystyle \frac{\bar{u}^{2,n+1}-\bar{u}^{2,n}}{k}-M\Delta \left(\frac{\bar{u}^{2,n+1}+\bar{u}^{2,n}}{2} \right)+f\left(\frac{\bar{u}^{2,n+1}+\bar{u}^{2,n}}{2}+\varphi(t_{n+1/2})\right)=\tilde{S}(t_{n+1/2}),\mbox{  in } \Omega,\\\\
B \bar{u}^{2,n}=0 \mbox{  on }  \partial \Omega,\\\\
\bar{u}^{2,0}=u_0-\varphi(0) \mbox{ in } \Omega.
\end{array}
\right. 
\end{equation} 
{For $j\geq 1$, we have }

\begin{equation}
\label{c2} 
\left\lbrace 
\begin{aligned} &\frac{\bar{u}^{2j+2,n+1}-\bar{u}^{2j+2,n}}{k}-\Lambda^j D\bar{u}^{2j,n+1/2}-M\Delta \left(E\bar{u}^{2j+2,n+1/2}-\Gamma^j E\bar{u}^{2j,n+1/2} \right)\\& +f\left(\frac{\bar{u}^{2j+2,n+1}+\bar{u}^{2j+2,n}}{2}-\Gamma^j E\bar{u}^{2j,n+1/2}+\varphi(t_{n+1/2}) \right)=\tilde{S}(t_{n+1/2}), \mbox{ in }\Omega, \mbox{ for }n\geq j+1,\\&
\quad \quad\quad\quad \quad \quad\quad\quad \quad\quad B \bar{u}^{2j+2,n}=0 \mbox{  on }  \partial \Omega,\\&
\quad \quad \quad \quad\quad\quad \quad \quad\quad\quad \bar{u}^{2j+2,0}=u_0-\varphi(0), \mbox{ in } \Omega&
\end{aligned}
\right. 
\end{equation}
where $\Gamma$ and $\Lambda$ are finite differences operators defined for each positive integer $j$, and $n\geq j$, by
\begin{equation}
\label{c2ab}
\Lambda^j u(t_n)=\sum_{i=1}^jc_{2i+1}k^{2i}(D_+D_-)^iu(t_n)=\sum_{i=1}^j\sum_{l=0}^{2i}c_{2i+1}(-1)^l{{2i}\choose{l}}u(t_{n+i-l}),
\end{equation}
and
\begin{equation}
\label{c2ac}
\Gamma^j u(t_n)=\sum_{i=1}^jc_{2i}k^{2i}(D_+D_-)^iu(t_n)=\sum_{i=1}^j\sum_{l=0}^{2i}c_{2i}(-1)^l{{2i}\choose{l}}u(t_{n+i-l}).
\end{equation}

The scheme (\ref{c1}) has unknowns $\left\lbrace \bar{u}^{2,n}\right\rbrace_{n=1}^N $ corresponding to approximations of $\bar{u}(t_n)$, expected to be of order 2 of accuracy. For (\ref{c2}) the unknowns are $\left\lbrace \bar{u}^{2j+2,n}\right\rbrace_{n=j+1}^N$, expected to be of order $2j+2$, while $\left\lbrace \bar{u}^{2j,n}\right\rbrace_{n=j}^N$ is supposed known from the preceding stage. To avoid computing approximate solution of (\ref{a1}) for $t<0$, the scheme (\ref{c2}) is used only for $n\geq j$. For the starting values, $0\leq n\leq j-1$, we consider the scheme
\begin{equation}
\label{c2b} 
\left\lbrace 
\begin{aligned} D\bar{u}^{2j+2,n+1/2}-&\frac{1}{2j+1}\bar{\Lambda}^j D \bar{\bar{u}}^{\,2j,n_j+1/2}-M\Delta \left(  E\bar{{u}}^{\,2j+2,n+1/2}-\bar{\Gamma}^j E\bar{\bar{u}}^{\,2j,n_j+1/2} \right) \\& +f\left(E\bar{{u}}^{\,2j+2,n+1/2}-\bar{\Gamma}^j E\bar{\bar{u}}^{\,2j,n_j+1/2}+\varphi(t_{n+1/2})  \right)=\tilde{S}(t_{n+1/2}),\\&
B \bar{u}^{2j+2,n}=0 \mbox{  on }  \partial \Omega,\\&
\quad \bar{u}^{2j+2,0}=u_0-\varphi(0), \mbox{ in } \Omega,
\end{aligned}
\right. 
\end{equation}
where we set $n_j=(2j+1)n+j$,
%$$&=\frac{1}{2j+1}\sum_{i=1}^jc_{2i+1}^jk_j^{2i}D(D_+D_-)^i\bar{u}^{2j,(2j+1)n+j+1/2}$$
\begin{equation}
\begin{aligned}
\label{c2c}
\frac{1}{2j+1}\bar{\Lambda}^j D \bar{\bar{u}}^{\,2j,(2j+1)n+j+1/2}=k^{-1}\sum_{i=1}^j\sum_{l=0}^{2i+1}c_{2i+1}^j(-1)^l{{2i+1}\choose{l}}\bar{\bar{u}}^{\,2j,(2j+1)n+j+i-l+1},
\end{aligned}
\end{equation}and
\begin{equation}
\label{c2d}
\begin{aligned}
\bar{\Gamma}^j\bar{\bar{u}}^{\,2j,(2j+1)n+j}=\sum_{i=1}^j\sum_{l=0}^{2i}c_{2i}^j(-1)^l{{2i}\choose{l}}\bar{\bar{u}}^{\,2j,(2j+1)n+j+i-l}.
\end{aligned}
\end{equation}
This scheme is built from (\ref{01}) and (\ref{02}), for $a=t_n$ and $b=t_{n+1}$. $\left\lbrace \bar{\bar{u}}^{\,2,n}\right\rbrace_{n=1}^N $ is computed from (\ref{c1}) with time the step $k/3$ instead of $k$. Similarly, $\left\lbrace \bar{\bar{u}}^{\,2j,n}\right\rbrace_{n=j}^N $, $j\geq 2$, is computed from the scheme (\ref{c2}) with the time step $k/(2j+1)$ instead of $k$.

To prove the existence of a solution for the schemes (\ref{c1}) and (\ref{c2}), we need the following lemma which is also indispensable for the proof of convergence of the semi-discrete solutions from (\ref{c2}) in the next section.
\begin{lemma}\label{lem:1} Let $k>0$ such that $k\max\left\lbrace  |\tau (0)|,\mu_0\right\rbrace  \leq 1/4$, $v \in L^2(\Omega)$ and $\sigma\in H^2(\Omega)$. Suppose that one of the following four conditions is satisfied:
\begin{itemize}
\item[(i)] $B$ denotes Dirichlet boundary conditions;
\item[(ii)] $B$ denotes Neunann boundary conditions  with $\sigma =0$;
\item[(iii)] $B$ denotes Neumann boundary conditions, and $2\leq q<9$ in (\ref{a3ab});
\item[(iv)] $B$ denotes Neumann boundary conditions, $9\leq q<18$ in (\ref{a3ab}), and $\sigma\in H^3(\Omega)$.
\end{itemize}
Then the elliptic problem
\begin{equation}
\label{c3}
u-kM \Delta u+kf(u+\sigma)=v ~~~\mbox{ in } \Omega,
\end{equation}

\begin{equation}
\label{c4} B u=0~~\mbox{  on }  \partial \Omega
\end{equation}has a solution $u\in H^2(\Omega)$ 
satisfying the inequality
\begin{equation}
\label{c5}
 \| u\|_2 \leq C\left(k^{-1} \| v-u\|+\| u\|_1+\|f(\sigma)\|_1+\|\sigma\|_2 \right)^\nu,
\end{equation}
where $\nu\geq 1$ is a real taking the value 1 in cases (i) and (ii), $C$ is a constant depending only on the positive definite matrix $M$ and $\Omega$, $\sigma$. The Neumann boundary conditions may be replaced by mixed Dirichlet-Neumann boundary conditions.
\end{lemma}

\begin{proof}The existence can be deduced from the Schaefer fixed point theorem \cite[ p. 541]{evans2010partial}. In fact, given $u\in H^2(\Omega)$, the problem
\begin{equation}
\label{c7}
w-kM \Delta w+kf(u+\sigma)=v ~~~\mbox{ in } \Omega,
\end{equation}
\begin{equation}
\label{c8} B w=0~~\mbox{  on }  \partial \Omega,
\end{equation}
has a unique solution $w\in H^2(\Omega)$ (see \cite[p.336]{evans2010partial} or \cite[ Thm 3.10]{ern2013theory}). Consider the nonlinear mapping 
$$A\,:\,  H^2(\Omega) \longrightarrow H^2(\Omega),$$
which maps $u\in H^2(\Omega)$ to the unique solution $w=A[u]$ of (\ref{c7})-(\ref{c8}). It is enough to prove that $A$ is continuous, compact, and that the set 
\begin{equation}
\label{c8b} \displaystyle \Sigma =\left\lbrace u\in H^2(\Omega)\,|\,u=\lambda A[u],  \mbox{ for some } \lambda \in [0,1] ~ \right\rbrace
\end{equation}is bounded.

\vspace*{.5cm}
\noindent
(I) ~~~The mapping $A$ is continuous. Indeed, let $\left\lbrace  u_m \right\rbrace  _{m=1}^{\infty} $ in $H^2(\Omega)$ that converges to $u\in H^2(\Omega)$. For each $m=1,2,\cdots$, let $w_m=A[u_m]$ and $w=A[u]$. Then $w-w_m$ belongs to $H^2(\Omega)$ and satisfies the equation

\begin{equation}
\label{c9}(w-w_m)-k M\Delta (w-w_m)+k\left(  f(u+\sigma)-f(u_m+\sigma) \right) =0 ~~~\mbox{ in } \Omega.
\end{equation}

$$B(w-w_m)=0\,\mbox{ on } \partial \Omega.$$
The inner product of (\ref{c9}) with $w-w_m$, taking into account the boundary condition,  yields

\begin{equation}
\label{c10}\|w-w_m\|^2+\gamma k\|\nabla (w-w_m)\|^2+k\left(  f(u+\sigma)-f(u_m+\sigma),w-w_m \right)\leq 0. 
\end{equation} We can write,
$$f(u(x)+\sigma(x))-f(u_m(x)+\sigma(x))=\int_0^1df\left( u(x)+\sigma(x)-\xi (u(x)-u_m(x)))(u(x)-u_m(x)\right) d\xi. $$
Since $u_m  \longrightarrow u$ in $H^2(\Omega)$ and $H^2(\Omega)\hookrightarrow C^0(\overline{\Omega})$, there exists a positive integer $m_0$ such that $m\geq m_0$ implies
\begin{equation}
\label{c10b}\max_{x\in \overline{\Omega}}\vert u(x)-u_m(x)\vert \leq c_2\|u-u_m\|_2 \leq 1,
\end{equation}
where $c_2$ is the  constant from the Sobolev embedding. It follows that
\begin{equation}
\label{c11}\vert f(u(x)+\sigma(x))-f(u_m(x)+\sigma(x))\vert \leq \beta \vert u(x)-u_m(x)\vert,
\end{equation}
where
$$\beta=\max_{|y| \leq 1+c_2( \| u\|_2+\|\sigma\|_2)}\vert df(y)\vert .$$
Therefore, by Cauchy-Schwartz inequality we have
$$
k\left\vert \left( f(u+\sigma)-f(u_m+\sigma),w-w_m  \right) \right \vert \leq  k\beta\|u-u_m\|\|w-w_m\| \leq \frac{(k\beta)^2}{2} \|u-u_m\|^2+\frac{1}{2}\|w-w_m\|^2 .$$
The last inequality substituted into (\ref{c10}) yields
$$\|w-w_m\|^2+2\gamma k\|\nabla (w-w_m)\|^2\leq (k\beta)^2 \|u-u_m\|^2.$$
It follows that $w_m\rightarrow w$ in $H^1(\Omega)$ when $m\rightarrow +\infty$. On the other hand, elliptic regularity results applied to the identity (\ref{c9}) yields, owing to (\ref{c11}) and the last inequality,
$$\|w-w_m\|_2\leq C\left( k^{-1}\|w-w_m\| +\|f(u+\sigma)-f(u_m+\sigma)\|\right)\leq 2\beta C\|u-u_m\| \rightarrow 0\mbox{ as } m\rightarrow +\infty. $$
Whence $\left\lbrace   w_m \right\rbrace  _{m=1}^{+\infty} $ converges to $w$ in $H^2(\Omega)$, and the continuity of the mapping $A$ follows.

\vspace*{.5cm}
\noindent
(II)\,\,The mapping A is compact. Indeed, given a bounded sequence $\left\lbrace  u_m \right\rbrace _{m\in \mathbb{N}} $ in  $H^2(\Omega)$, from the compact embedding $H^2(\Omega) \hookrightarrow L^2(\Omega)$ we can extract a  subsequence $\left\lbrace  u_{m_j} \right\rbrace _{j\in \mathbb{N}} $ that converges to $u$ strongly in $L^2(\Omega)$ and weakly in $H^2(\Omega)$.  The subsequence $\left\lbrace  u_{m_j} \right\rbrace _{j\in \mathbb{N}}$ is then bounded in $H^2(\Omega)$. Let
$$ \kappa =\sup_{m\in \mathbb{N}}\|u_m\|_2 ~~\mbox{ and }~~\beta '=\displaystyle\max_{|y|\leq c_2(\kappa +\|u\|_2+\|\sigma\|_2)}|df(y)|.$$
Therefore, proceeding exactly as in part (I), substituting $m$ by $m_j$, the inequality (\ref{c10b}) by
$$\max_{x\in \overline{\Omega}}\vert u_{m_j}(x)\vert \leq c_2\sup_{m\in \mathbb{N}}\|u_{m_j}\|_2=c_2\kappa,$$ and $\beta$ by $\beta '$ in (\ref{c11}), we deduce that $w_{m_j}=A[u_{m_j}] \rightarrow w$ strongly in $H^2(\Omega)$. Hence $A$ is compact.

\vspace*{.5cm}
\noindent
(III)\,\, We prove that the set $\displaystyle \Sigma$ is bounded, proceeding in two different ways, depending on the boundary condition. \vspace{.3cm}

\noindent 
a)\,\,  First we suppose that $B$ denotes Dirichlet boundary conditions (case (i)), or $B$ denotes Neumann boundary conditions  with $\sigma =0$ (case (ii)).

Let $ u\in H^2(\Omega)$ such that $u=\lambda A[u]$ for some $\lambda \in (0,1]$. Then $u$ satisfies 
\begin{equation}
\label{c12}
u-kM\Delta u+\lambda kf(u+\sigma)=\lambda v ~~\mbox{ in } \Omega,
\end{equation}
\begin{equation}
\label{c12b} B u=0 ~~\mbox{ on } \partial \Omega.
\end{equation}
By elliptic regularity we have
\begin{equation}
\label{c13} \|u\|_2\leq C\|k^{-1}(\lambda v-u)-\lambda f(u+\sigma)\|=C\|M\Delta u\| .
\end{equation}The inner product of (\ref{c12}) with $u+\sigma$ yields

\begin{equation}
\label{c13b}\|u+\sigma\|^2-k\int_\Omega M\Delta (u+\sigma)\cdot (u+\sigma)\,dx+\lambda k\int_{\Omega}f(u+\sigma)\cdot \left( u+\sigma\right)\, dx= \int_{\Omega}\left( \lambda v+\sigma-kM\Delta \sigma \right) \cdot (u+\sigma)dx.
\end{equation}
We can write
$$M\Delta (u+\sigma)\cdot (u+\sigma)=\sum_{i,j=1}^JM_{ij}\nabla\cdot \left[(u_i+\sigma_i)\nabla (u_j+\sigma_j) \right]-\sum_{l=1}^d\left[ M\frac{\partial (u+\sigma)}{\partial x_l}\right] \cdot\frac{\partial (u+\sigma)}{\partial x_l} $$and deduce from the Divergence Theorem  that
\begin{equation}
\label{c13c}\begin{aligned}-k\int_\Omega M\Delta (u+\sigma)\cdot (u+\sigma)\,dx &=-k\sum_{i,j=1}^J\int_{\partial\Omega}\, M_{ij}\sigma_i\nabla (u_j+\sigma_j)\cdot\nu\, d\mathbf{S}+k\sum_{l=1}^d\left( M\frac{\partial (u+\sigma)}{\partial x_l},\frac{\partial (u+\sigma)}{\partial x_l}\right) \\&\geq\gamma k \|\nabla (u+\sigma)\|^2-k\Vert M\Vert_F \Vert \sigma\Vert_1\Vert u+\sigma\Vert_2 
\\&\geq\gamma k\|\nabla (u+\sigma)\|^2-\frac{1}{2}k^2\varepsilon_1^2\Vert u+\sigma\Vert_2^2-\frac{1}{2\varepsilon_1^2}\Vert M\Vert_F^2 \Vert \sigma\Vert_1^2,
\end{aligned}
\end{equation}
where $\nu$ is the outward pointing normal vector field to $\partial \Omega$, $\gamma$ is the smallest eigenvalue of the positive definite matrix $M$, $\varepsilon_1$ is an arbitrary positive real, and $\Vert\,\cdot\,\Vert_F $ is the Frobenius norm. Without loss of generality we suppose that $f(0)=0$, otherwise we change $f$ by $\tilde{f}=f-f(0)$ and $v$ by $\tilde{v}=v-kf(0)$. The monotonicity condition (\ref{a2}) combined with the hypothesis of the lemma yields
 \begin{equation}
\label{c14}
\lambda k\int_{\Omega}f(u+\sigma)\cdot \left( u+\sigma\right)\, dx \geq \alpha\lambda k\|u+\sigma\|_{L^q(\Omega)}^q+\lambda k\tau(0)\|u+\sigma\|^2\geq \alpha\lambda k\|u+\sigma\|_{L^q(\Omega)}^q-\frac{1}{4}\|u+\sigma\|^2, 
\end{equation}
for all $\lambda\in (0,1]$. From Cauchy-Schwartz inequality and the Cauchy inequality with $\varepsilon =1$, we have
\begin{equation}
\label{c14b} \int_{\Omega}\left(\lambda v+\sigma-kM\Delta \sigma \right) \cdot (u+\sigma)\,dx \leq  \|\lambda v+\sigma-kM\Delta \sigma\|^2+\frac{1}{4}\|u+\sigma\|^2.
\end{equation}
Substituting the last three inequalities in (\ref{c13b}), we deduce that 
\begin{equation}
\label{c15}
\|u+\sigma\|^2+2 \gamma k\|\nabla (u+\sigma)\|^2\leq 2 \| \lambda v+\sigma-kM\Delta \sigma\|^2+k^2\varepsilon_1^2\Vert u+\sigma\Vert_2^2+{\varepsilon_1^{-2}}\Vert M\Vert_F^2 \Vert \sigma\Vert_1^2.
\end{equation}
On the other hand, the inner product of (\ref{c12}) with $-\Delta (u+\sigma)$ yields 
\begin{equation}
\label{c15b}
\gamma \|\Delta (u+\sigma)\|^2\leq -k^{-1}\int_{\Omega} (\lambda v-u-kM                                      \Delta \sigma)\cdot \Delta (u+\sigma)dx+\int_{\Omega} \lambda f(u+\sigma)\cdot \Delta \left( u+\sigma\right) \,dx.
\end{equation}
We can write
$$f(u+\sigma)\cdot \Delta \left( u+\sigma\right) =\sum_{i=1}^J\nabla \cdot \left[ f_i(u+\sigma)\nabla (u_i+\sigma_i)\right]-\sum_{i=1}^d\left[ df(u+\sigma)\left( \frac{\partial (u+\sigma)}{\partial x_i}\right)\right] \cdot\frac{\partial (u+\sigma)}{\partial x_i}, $$
and deduce from (\ref{a3}), Dirichlet boundary conditions  and the Divergence Theorem that
$$
\begin{aligned}
\int_{\Omega}f(u+\sigma)\cdot \Delta \left(u+\sigma \right)\, dx&=\sum_{i=1}^J\int_{\partial \Omega}f_i(\sigma)\nabla (u_i+\sigma_i)\cdot\nu\, d\mathbf{S}\\& -\sum_{i=1}^d\int_{\Omega}\left[ df(u+\sigma)\left( \frac{\partial (u+\sigma)}{\partial x_i}\right)\right] \cdot\frac{\partial (u+\sigma)}{\partial x_i} dx \\&\leq \|f(\sigma)\|_1\|u+\sigma\|_2 +
\mu_0 \sum_{i=1}^d\left \Vert \frac{\partial( u+\sigma)}{\partial x_i}\right \Vert^2
\\&\leq \frac{1}{2}\varepsilon_2^2\|u+\sigma\|_2^2+\frac{1}{2\varepsilon_2^2}\|f(\sigma)\|_1^2+\mu_0 \|\nabla (u+\sigma)\|^2,
\end{aligned}
$$
where $\varepsilon_2$ is an arbitrary positive real. For Neumann boundary conditions with $\sigma=0$, the integral on $\partial\Omega$ vanishes,  and the right side of the later inequality may be replaced by $\mu_0 \|\nabla (u+\sigma)\|^2$. By Cauchy-Schwartz inequality and the Cauchy inequality with $\varepsilon =1/(2\gamma) $ we have 
$$ 
\begin{aligned}
\left \vert k^{-1}\int_{\Omega} (\lambda v-u-kM\Delta \sigma)\cdot \Delta(u+\sigma)dx \right\vert&\leq k^{-1}\|\lambda v-u-kM\Delta \sigma\|\|\Delta (u+\sigma)\|\\&\leq \frac{1}{2\gamma k^2} \|\lambda v-u-kM\Delta \sigma\|^2+\frac{\gamma}{2}\|\Delta (u+\sigma)\|^2.\end{aligned}$$
Substituting the last two inequalities in (\ref{c15b}), we obtain
$$\gamma^2 \|\Delta (u+\sigma)\|^2 \leq k^{-2} \|\lambda v-u-kM\Delta \sigma\|^2+\lambda\gamma\varepsilon_2^2\|u+\sigma\|_2^2+(\lambda\gamma/\varepsilon_2^2)\|f(\sigma)\|_1^2+2\lambda\gamma\mu_0 \|\nabla (u+\sigma)\|^2.
$$
Therefore, 
\begin{equation*}
\begin{aligned}
 \| M\Delta (u+\sigma)\|^2 \leq &\left(\Vert M\Vert_F/\gamma \right)^2\left(k^{-2} \|\lambda v-u-kM\Delta \sigma\|^2+\gamma\varepsilon_2^2\|u+\sigma\|_2^2+(\gamma/\varepsilon_2^2)\|f(\sigma)\|_1^2\right. \\&\left. +2\gamma\mu_0 \|\nabla (u+\sigma)\|^2 \right) 
 \end{aligned}
\end{equation*} since $0\leq \lambda \leq 1$, and we deduce from (\ref{c13}) and the triangle inequality that 
\begin{equation}
\label{c16} \| u\|_2 \leq C\left[ k^{-1} \|\lambda v-u\|+\sqrt{\gamma}\varepsilon_2\|u\|_2+\sqrt{\gamma}\varepsilon_2^{-1}\|f(\sigma)\|_1+\sqrt{2\gamma\mu_0} \|\nabla (u+\sigma)\|+(1+\sqrt{\gamma} \varepsilon_2)\|\sigma\|_2\right] . 
\end{equation} 
Fixing $\varepsilon_1=\varepsilon_2=1/(3C+2\sqrt{\gamma}C)$, the last inequality together with (\ref{c15}) and the inequality $k\mu_0\leq 1/4$ yields, 
$$\|u\|_2\leq C\left( k^{-1}\|v\|+k^{-1}\|\sigma\|_1+\|f(\sigma)\|_1+\|\sigma\|_2\right) ,$$
and it follows that $\Sigma$ is bounded. From (i)-(iii) we deduce by the Schaefer fixed point theorem that (\ref{c3})-(\ref{c4}) has a solution  $u\in H^2(\Omega)$ and (\ref{c5}) follows, taking $\lambda =1$ in (\ref{c16}).

\vspace{.3cm}
\noindent
b)\,\, We suppose that $B$ denotes Neumann boundary conditions and $f$ satisfies satisfies (\ref{a3ab}), cases (iii)-(iv). Then, for each $u\in H^2(\Omega)$ such that $u=\lambda A[u]$, for some $\lambda\in (0,1]$, the function $u$ satisfies (\ref{c12}), and inequality (\ref{c13c}) becomes
\begin{equation}
\label{c17}\begin{aligned}-k\int_\Omega M\Delta (u+\sigma)\cdot (u+\sigma)\,dx &=-k\sum_{i,j=1}^J\int_{\partial\Omega}\, M_{ij}(\sigma_i+u_i)\nabla \sigma_j\cdot\nu\, d\mathbf{S}+k\sum_{l=1}^d\left( M\frac{\partial (u+\sigma)}{\partial x_l},\frac{\partial (u+\sigma)}{\partial x_l}\right) \\&\geq\gamma k \|\nabla (u+\sigma)\|^2-k\Vert M\Vert_F \Vert \sigma\Vert_2\Vert u+\sigma\Vert_1 
\\&\geq\frac{1}{2} \gamma k\|\nabla (u+\sigma)\|^2-\frac{1}{4}\Vert u+\sigma\Vert^2-\left( k^2+\gamma^{-1} k/2\right)\Vert M\Vert_F^2 \Vert \sigma\Vert_2^2.
\end{aligned}
\end{equation}
Since identity (\ref{c13b}) and inequalities (\ref{c14})-(\ref{c14b}) do not depend on boundary conditions, according to (\ref{c17}), inequality (\ref{c15}) becomes
\begin{equation}
\label{c18}
\gamma k\|\nabla (u+\sigma)\|^2+4\lambda \alpha k\|u+\sigma\|_{L^q(\Omega)}^q \leq 4 \| \lambda v-u-kM\Delta \sigma\|^2+\left(4k^2+2\gamma^{-1} k \right) \Vert M\Vert_F^2 \Vert \sigma\Vert_2^2.
\end{equation}
% and
%\begin{equation}
%\label{c18b}
%{\color{red}\|u+\sigma\|^2 \leq 4 \| \lambda v+\sigma-kM\Delta \sigma\|^2+\left(4k^2+2\gamma^{-1} k \right) \Vert M\Vert_F^2 \Vert \sigma\Vert_2^2.}
%\end{equation}
We can write
\begin{equation}
\label{c18b}\begin{aligned}
f(u+\sigma)\cdot \Delta u =\sum_{i=1}^J\nabla \cdot &\left[ f_i(u+\sigma)\nabla u_i\right]-\sum_{l=1}^d\left[ df(u+\sigma)\left( \frac{\partial (u+\sigma)}{\partial x_l}\right)\right] \cdot\frac{\partial (u+\sigma)}{\partial x_l}\\&+\sum_{l=1}^d\left[ df(u+\sigma)\left( \frac{\partial (u+\sigma)}{\partial x_i}\right)\right] \cdot\frac{\partial \sigma}{\partial x_l}.\end{aligned}
\end{equation}
%From (\ref{a3ab}) and the Sobolev embedding $H^2(\Omega)\hookrightarrow C^0(\bar{\Omega})$ we have, for almost each $x\in \Omega$,
%$$\begin{aligned} \left |\left[ df(u+\sigma)\left( \frac{\partial (u+\sigma)}{\partial x_l}\right)\right] \cdot\frac{\partial \sigma}{\partial x_l}\right |(x)&\leq \mu_1\left(\left | \frac{\partial (u+\sigma)}{\partial x_l}\right |\left | \frac{\partial \sigma_l}{\partial x_i}\right |+ \left | \frac{\partial (u+\sigma)}{\partial x_i}\right |\left | \frac{\partial \sigma_l}{\partial x_i}\right | \left |u+\sigma\right |^{q-2}\right)(x)\\&\leq C\left(\left | \frac{\partial (u+\sigma)}{\partial x_i}\right |\left | \frac{\partial \sigma_l}{\partial x_i}\right |+\left | \frac{\partial \sigma_l}{\partial x_i}\right | \left | \frac{\partial (u+\sigma)}{\partial x_i}\right |\Vert u+\sigma\Vert_2^\varepsilon \left |u+\sigma\right |^{q-2-\varepsilon}\right)(x),\end{aligned} $$
%for arbitrary $0\leq \varepsilon_3 <1$.
% Therefore, $(3/2)(q-2-\varepsilon_3)\leq q$, for $2\leq q<9$, and it follows from Sobolev embeddings $W^{1,6}(\Omega)\hookrightarrow C^0(\bar{\Omega})$, $H^2(\Omega)\hookrightarrow C^0(\bar{\Omega})$, $L^q(\Omega)\hookrightarrow L^{(3q_{\varepsilon_3}/2)}$, $q_{\varepsilon_3}=q-2-\varepsilon_3$,  
%Let $0<\varepsilon_3<1$ with $\varepsilon_3=q-2$ if $q<4$, and $q_3=q-2-\varepsilon_3$ if $4\leq q<9$. 
There exists a real $0\leq \varepsilon_3<1$ such that, for each of conditions (iii)-(iv), we have 
\begin{equation}
\label{c19}
\begin{aligned}\left |\int_\Omega\sum_{i=1}^d\left[ df(u+\sigma)\left( \frac{\partial (u+\sigma)}{\partial x_i}\right)\right]\right. &\left.  \cdot\frac{\partial \sigma}{\partial x_i}\,dx \right|\leq 
\varepsilon_4 \|u+\sigma\|_2^2+C_{4}\left(\|\nabla \sigma\|^2+\|\nabla (u+\sigma)\|^2+C_\sigma\|u+\sigma\|_{L^q(\Omega)}^{\frac{2q_3}{1-\varepsilon_3}} \right),\end{aligned}
\end{equation}  
where 
$$C_\sigma=\displaystyle \left\lbrace \begin{array}{ccc}
\|\sigma\|_2^{2/(1-\varepsilon_3)},& \mbox{ if } 2\leq q<9\\\\
\|\sigma\|_3^{2/(1-\varepsilon_3)},& \mbox{ if } \sigma\in H^3(\Omega)\, \mbox{ and }\,9\leq q<18,
  \end{array}\right.  $$
 $\varepsilon_4$ is an arbitrary positive real, $q_3=q-2-\varepsilon_3$, and $C_{4}$ is a constant depending only on $\varepsilon_4$, $\Omega$ and possibly on $\sigma$. In fact:

\begin{itemize}
\item for $2\leq q<9$, we choose $\varepsilon_3=\max\left\lbrace 0\,,\, q/3-2 \right\rbrace $ such that $1\leq 2q/(2+\varepsilon_3)\leq 6$, and it follows from (\ref{a3ab}), the Sobolev embeddings $ H^2(\Omega)\hookrightarrow  C^0(\bar{\Omega})$ and $ H^2(\Omega)\hookrightarrow W^{1,2q/(2+\varepsilon_3)}(\Omega)$ that
$$\begin{aligned}\left |\int_\Omega\sum_{i=1}^d \right. &\left. \left[ df(u+\sigma)\left( \frac{\partial (u+\sigma)}{\partial x_i}\right)\right]  \cdot\frac{\partial \sigma}{\partial x_i}\,dx \right|\\&\leq \mu_1\left(\|\nabla (u+\sigma)\|\|\nabla \sigma\|+\|u+\sigma\|_2^{\varepsilon_3}\|\nabla \sigma\|_{L^{2q/(2+\varepsilon_3)}(\Omega)}\|\nabla (u+\sigma)\|_{L^{2q/(2+\varepsilon_3)}(\Omega)}\|u+\sigma\|_{L^{q}(\Omega)}^{q_3} \right)
\\&\leq C\left(\|\nabla (u+\sigma)\|\|\nabla \sigma\|+\|u+\sigma\|_2^{\varepsilon_3+1} \|\sigma\|_2 \|u+\sigma\|_{L^{q}(\Omega)}^{q_3} \right).
  \end{aligned}$$
  This inequality yields (\ref{c19}) according to the Young inequality.

\item For $9\leq q<18$ and $\sigma \in H^3(\Omega)$, we have $(6/5)q_3\leq q$, and  it follows from (\ref{a3ab}) and the Sobolev embeddings $ H^2(\Omega)\hookrightarrow  C^0(\bar{\Omega})$, $ H^2(\Omega)\hookrightarrow W^{1,6}(\Omega)$, and  $L^q(\Omega)\hookrightarrow L^{6q_3/5}(\Omega)$ that
$$\begin{aligned}\left |\int_\Omega\sum_{i=1}^d\right. &\left.\left[ df(u+\sigma)\left( \frac{\partial (u+\sigma)}{\partial x_i}\right)\right]  \cdot\frac{\partial \sigma}{\partial x_i}\,dx \right|\\&\leq \mu_1\left(\|\nabla (u+\sigma)\|\|\nabla \sigma\|+\|u+\sigma\|_2^{\varepsilon_3}\| \sigma\|_3\|\nabla (u+\sigma)\|_{L^6(\Omega)}\|u+\sigma\|_{L^{6q_{3}/5}(\Omega)}^{q_3} \right)
\\&\leq C\left(\|\nabla (u+\sigma)\|\|\nabla \sigma\|+\|u+\sigma\|_2^{\varepsilon_3+1} \|\sigma\|_3 \|u+\sigma\|_{L^q(\Omega)}^{q_3} \right).
  \end{aligned}$$
  This inequality implies (\ref{c19}) from the Young inequality.
\end{itemize} 
Integrating (\ref{c18b}) on $\Omega$, we deduce from the homogeneous Neumann boundary conditions on $u$, and inequalities (\ref{a3}) and (\ref{c19}) that
$$\begin{aligned}\int_\Omega f(u+\sigma)\cdot \Delta u\,dx \leq \mu_0 \|u+\sigma\|^2+\varepsilon_4 \|u+\sigma\|_2^2+C_{4}\left(\|\nabla (u+\sigma)\|^2+\|\sigma\|^2_2+\|u+\sigma\|_{L^q(\Omega)}^{\frac{2q_3}{1-\varepsilon_3}} \right).\end{aligned} $$
Therefore, taking the inner product of (\ref{c12}) with $-\Delta u$, we deduce that
$$\begin{aligned}\gamma^2\| \Delta u\|^2 &\leq k^{-2}\|\lambda v-u\|^2+2\gamma\mu_0 \|u+\sigma\|^2+2\gamma\varepsilon_4\|u+\sigma\|_2^2\\&+2\gamma C_{4}\left(\|\nabla (u+\sigma)\|^2 +\|\sigma\|^2_2+\|u+\sigma\|_{L^q(\Omega)}^{\frac{2q_3}{1-\varepsilon_3}} \right),\end{aligned}$$
for any $0<\lambda\leq 1$. Since $\varepsilon_4$ is arbitrary positive, the conclusion follows from elliptic regularities,  proceeding as in part iii-(a). The case where $B$ is mixed Dirichlet-Neumann boundary conditions is deduced from iii-(a)-(b).  
\end{proof}

 The following theorem shows the existence of a solution for the schemes (\ref{c1}) and (\ref{c2}).
 
 \begin{theorem}
\label{theo:2} 
 Suppose that $u_0-\varphi(0)\in H^2(\Omega)$. Then, for each nonnegative integer $n$, the scheme (\ref{c1}) and (\ref{c2}) has a solution in $H^2(\Omega)$.
 \end{theorem}
 
 \begin{proof}Proceeding by induction, the proof is immediate from Lemma \ref{lem:1} for a suitable choice of the functions $u$ and $v$. For example,  multiplying the first equation in (\ref{c1}) by $k/2$, we deduce (\ref{c3})-(\ref{c4}) for $u=(\bar{u}^{2,n+1}+\bar{u}^{2,n})/2$, $\sigma=\varphi( t_{n+1/2})$, $v=k\bar{S}(t_{n+1/2})/2+\bar{u}^{2,n}$, and $k$ substituted by $k/2$. 
 \end{proof}

Hereafter we suppose that $\bar{u}^{2j,n} \in H^{r+1}(\Omega)$, for $1\leq j\leq p+1$ and each $n=0,1\cdots, N$. Convergence results for these semi-discrete solutions are proven in section \ref{sec:Semi}. 

\section{Convergence and order of accuracy of the semi-discrete solution}
\label{sec:Semi}
This section deals with the analysis of convergence of the time-semidiscrete solutions of problem (\ref{aa1}) given by formulae (\ref{c1}) and (\ref{c2}). We start by introducing a Deferred Correction Condition (DCC) for the semi-discrete solutions in time, as in the case of IVPs from the reference \cite{koyaguerebo2021arbitrary}, and prove in Theorem \ref{theo:3} that the DCC is a sufficient condition for each semi-discrete solution $ \left\lbrace \bar{u}^{2j,n}\right\rbrace_n$ to be corrected with an increment of order 2 of accuracy. The DCC is proven for the implicit midpoint scheme (\ref{c1}) in Lemma \ref{lem:2} which, together with Theorem \ref{theo:3},  constitutes the pivot for the proof of convergence of the semi-discrete solutions $\left\lbrace \bar{u}^{2j,n}\right\rbrace_n$, $j\geq 2$. We have the following definition:

\begin{defn}[Deferred Correction Condition]  Let $u$ be the exact solution of (\ref{a1}),  and $\bar{u}=u-\varphi$. For a positive integer $j$, a sequence $\left\lbrace  \bar{u}^{2j,n}\right\rbrace _n\subset H^1(\Omega)$ of approximations of $\bar{u}$ on the uniform partition $0=t_0<t_1<\cdots <t_N=T$, $t_n=nk$, is said to satisfy the Deferred Correction Condition (DCC) for the implicit midpoint rule  if  $\left\lbrace   \bar{u}^{2j,n}\right\rbrace _n$ approximates $\bar{u}(t_n)$ with order $2j$ of accuracy in time, and for $n=1,2,...,N-2$ we have
\begin{equation} 
\label{d1}
\|(D_+D_-)D(\bar{u}^{2j,n+1/2}-\bar{u}(t_{n+1/2}))\|+\|D_+D_-(\bar{u}^{2j,n+1}-\bar{u}(t_{n+1}))\|\leq C k^{2j},
\end{equation}for each time steps $k\leq k_1$, where $k_1>0$ is fixed and $C$ is a constant independent from $k$.
\end{defn}

\begin{remark}
\label{rmk:4} Condition (\ref{d1}) is equivalent to
\begin{equation}
\label{cc1}
\left \| \Gamma^j\left(\bar{u}^{2j,n}-\bar{u}(t_n) \right) \right \|\leq Ck^{2j+2},
\end{equation}
and
\begin{equation}
\label{cc2}
\left \| (\Lambda^j -\Gamma^j)D\left(\bar{u}^{2j,n+1/2}-\bar{u}(t_{n+1/2}) \right) \right \|\leq Ck^{2j+2},
\end{equation}
for $n=j,j+1,\cdots , N-j$. This is due to the transforms
$$k^{2i}\left(D_+D_-\right)^i\left(\bar{u}^{2j,n}-\bar{u}(t_n) \right)=k^2\sum_{l=0}^{i-1}(-1)^l{{2i-2}\choose{l}}D_+D_-\left(\bar{u}^{2j,n}-\bar{u}(t_n) \right), $$
and
$$
\begin{aligned}
k^{2i} \left( D_+D_- \right)^iD &\left( \bar{u}^{2j,n+1/2}-\bar{u}(t_{n+1/2}) \right) &=k^2\sum_{l=0}^{i-1}(-1)^l{{2i-2}\choose{l}}\left( D_+D_-\right) D \left( \bar{u}^{2j,n+1/2}-\bar{u}(t_{n+1/2})\right) . \end{aligned}$$
\end{remark}
The following theorem gives a sufficient condition for the semi-discrete schemes in time to converge with the expected order of accuracy.

\begin{theorem}
\label{theo:3}
Let $j$ be a positive integer and $\left\lbrace   \bar{u}^{2j,n}\right\rbrace _n$ a sequence of approximations of $\bar{u}=u-\varphi$, on the discrete points $t_0=0<t_1<\cdots <t_N=T$, satisfying DCC for the implicit midpoint rule. Suppose that $k< k_1$, and that $\bar{u}^{2j+2,1},...,\bar{u}^{2j+2,j}$ are given and satisfy 
\begin{equation}
\label{d2}
\|\bar{u}^{2j+2,n}-\bar{u}(t_n)\|\leq Ck^{2j+2} ,~~\mbox{ for }~n=0,1,...,j.
\end{equation}
Then the sequence $\left\lbrace   \bar{u}^{2j+2,n}\right\rbrace _{n\geq j}$, solution of the scheme (\ref{c2}) built from $\left\lbrace  \bar{u}^{2j,n}\right\rbrace _n$, approximates $\bar{u}$ with order $2j+2$ of accuracy in time, and we have, for $n=0,1,\cdots,N$,
\begin{equation}
\label{d3}
\Vert \bar{u}^{2j+2,n}-\bar{u}(t_n)\Vert+ \left( \gamma k\sum_{i=j}^n\Vert \nabla \widehat{\Theta}^{2j+2,i} \Vert^2\right)^\frac{1}{2}  \leq C k^{2j+2},
\end{equation}
where
\begin{equation}
\label{d5}
\Theta^{2j+2,n}=\left(  \bar{u}^{2j+2,n}-\bar{u}(t_{n}) \right) -\Gamma^j \left( \bar{u}^{2j,n}-\bar{u}(t_n)\right),
\end{equation}
 and $C$ is a constant depending only on $j$, $T$, $M$, $u\in C^{2j+3}\left( [0,T],H^2(\Omega)\right)$, a Lipschitz constant on $f$ and the DCC constant. 
\end{theorem}

\begin{proof}Combining (\ref{c2}) and (\ref{aa1}), we obtain the identity
\begin{equation}
\label{d4}
\begin{split}
f\left( E\bar{u}^{2j+2,n+1/2}-\Gamma^j E\bar{u}^{2j,n+1/2}+\varphi(t_{n+1/2})\right)-f\left(E\bar{u}(t_{n+1/2})-\Gamma^j E\bar{u}(t_{n+1/2})+\varphi(t_{n+1/2})\right)\\+D\Theta^{2j+2,n+1/2}
-M\Delta \widehat{\Theta}^{2j+2,n+1}=\sigma^{2j+2,n+1/2}+ (\Lambda^j -\Gamma^j)D(\bar{u}^{2j,n+1/2}-\bar{u}(t_{n+1/2})),
\end{split}
\end{equation}
$$B \widehat{\Theta}^{2j+2,n+1}=0 \,\, \mbox{ on } \partial \Omega,$$
where
\begin{equation*}
\begin{aligned}
\sigma^{2j+2,n+1/2}=&\bar{u}'(t_{n+1/2})- D\bar{u}(t_{n+1/2})+\Lambda^j D\bar{u}(t_{n+1/2})-M\Delta\left( \bar{u}(t_{n+1/2})-E\bar{u}(t_{n+1/2})+\Gamma^j E\bar{u}(t_{n+1/2})\right)\\&+f(\bar{u}(t_{n+1/2})+\varphi(t_{n+1/2}))-f\left( E\bar{u}(t_{n+1/2})-\Gamma^j E\bar{u}(t_{n+1/2})+\varphi(t_{n+1/2})\right).
\end{aligned}
\end{equation*}
The inner product of (\ref{d4}) with $\widehat{\Theta}^{2j+2,n+1}$, taking into account the monotonicity condition (\ref{a2}) and the fact that $B \widehat{\Theta}^{2j+2,n+1}=0$ on $\partial \Omega$, yields
\begin{equation}
\label{d6}
\begin{aligned}
&(D\Theta^{2j+2,n+1/2},\widehat{\Theta}^{2j+2,n+1}) +\gamma \|\nabla \widehat{\Theta}^{2j+2,n+1}\|^2\leq\tau(E\bar{u}(t_{n+1/2})-\Gamma^j E\bar{u}(t_{n+1/2})+\varphi(t_{n+1/2}))\|\widehat{\Theta}^{2j+2,n+1}\|^2 \\&
 +\left( \sigma^{2j+2,n+1/2}+ (\Lambda^j -\Gamma^j)D(\bar u^{2j,n+1/2}-\bar u(t_{n+1/2})),\widehat{\Theta}^{2j+2,n+1}\right).
 \end{aligned} \end{equation}
From the central finite differences formulae (\ref{b6})-(\ref{b7}) and the mean value theorem we have 
$$\|\sigma^{2j+2,n+1/2} \|\leq Ck^{2j+2},$$
where $C$ is a constant depending only on a Lipschitz condition on $f$ and the norm of $u$ as element of $C^{2j+3}\left( [0,T],H^2(\Omega)\right)$, and there exists $0<k_2\leq k_1$ such that $ k\leq k_2$ implies that 
$$
\begin{aligned}\Vert E\bar{u}(t_{n+1/2})-\Gamma^j E\bar{u}(t_{n+1/2})+\varphi(t_{n+1/2})\Vert_\infty &\leq 
\Vert  \bar{u}(t_{n+1/2})-E\bar{u}(t_{n+1/2})+\Gamma^j E\bar{u}(t_{n+1/2}) \Vert_\infty+\|{u}(t_{n+1/2})\|_\infty \\&\leq 1+\|{u}\|_{L^\infty (Q_T)},\end{aligned}$$
where $Q_T=\Omega \times (0,T)$. It follows that, for $k\leq k_2$,
$$\left \Vert \tau(E\bar{u}(t_{n+1/2})-\Gamma^j E\bar{u}(t_{n+1/2}))+\varphi(t_{n+1/2}) \right \Vert_\infty \leq \max_{|y|\leq 1+\|{u}\|_{L^\infty (Q_T)}}|\tau (y)|=:\mu.  $$
On the other hand, from the DCC we immediately have
$$\|(\Lambda^j -\Gamma^j)D(\bar{u}^{2j,n+1/2}-\bar{u}(t_{n+1/2})\| \leq C k^{2j+2} .$$
Substituting the last inequalities in (\ref{d6}), taking into account the identity
$$\left(  D\Theta^{2j+2,n+1/2},\widehat{\Theta}^{2j+2,n+1} \right)=\frac{1}{2k}\left(\|\Theta^{2j+2,n+1}\|^2-\|\Theta^{2j+2,n}\|^2  \right), $$
we deduce that 
\begin{equation}
\label{d7}
\begin{aligned}
\|\Theta^{2j+2,n+1}\|^2-\|\Theta^{2j+2,n}\|^2  +2k\gamma \|\nabla \widehat{\Theta}^{2j+2,n+1}\|^2  \leq Ck^{2j+3}\|\widehat{\Theta}^{2j+2,n+1}\|+2k \mu \|\widehat{\Theta}^{2j+2,n+1}\|^2.
 \end{aligned} 
 \end{equation}
This inequality yields 
\begin{equation*}
\|\Theta^{2j+2,n+1}\|^2-\|\Theta^{2j+2,n}\|^2 \leq Ck^{2j+3}\|\widehat{\Theta}^{2j+2,n+1}\|+2k \mu \|\widehat{\Theta}^{2j+2,n+1}\|^2, 
 \end{equation*}and, for $\mu k<2$, we deduce from the inequality
$$\left \|\widehat{\Theta}^{2j+2,n+1}\right \|\leq \frac{1}{2}\left(\left \|\Theta^{2j+2,n+1}\right \|+\left \|\Theta^{2j+2,n}\right \|  \right) $$
that
$$\left \Vert \Theta^{2j+2,n+1} \right \Vert \leq C\frac{k^{2j+3}}{2-\mu k}+\frac{2+\mu k}{2-\mu k}\left \Vert \Theta^{2j+2,n} \right \Vert .$$ It follows by induction on $n$ that 
$$\left \Vert \Theta^{2j+2,n} \right \Vert \leq C\frac{1}{2-\mu k}\left( \frac{2+\mu k}{2-\mu k}\right)^{n-j-1} k^{2j+2}+\left( \frac{2+\mu k}{2-\mu k}\right)^{n-j}\left \Vert \Theta^{2j+2,j} \right \Vert .$$
From the hypothesis (\ref{d2}) and the DCC we have 
\begin{equation}
\label{d8}
\left \Vert \Theta^{2j+2,j} \right \Vert \leq \left \Vert\bar u^{2j+2,j}-\bar u(t_j)\right \Vert+ \left \Vert \Gamma^j \left( \bar u^{2j,j}-\bar u(t_j) \right)  \right \Vert\leq Ck^{2j+2},
\end{equation}
where $C$ is a constant independent from $k$. Moreover, the sequence $\left\lbrace \left( \frac{2+\mu k}{2-\mu k}\right)^{n}  \right\rbrace_n $  is bounded above by  $\exp(2\mu T/(2-\varepsilon))$, for $0\leq \mu k \leq \varepsilon <2$. Whence 

\begin{equation}
\label{d8b}
\Vert \Theta^{2j+2,n} \Vert \leq Ck^{2j+2}.
\end{equation}
Finally, by the triangle inequality, the identity (\ref{d5}) and the DCC, we have 
\begin{equation}
\label{d8c}
\Vert \bar u^{2j+2,n}-\bar u(t_n)\Vert \leq C k^{2j+2}+\left \Vert \Gamma^j(\bar u^{2j,n}-\bar u(t_n))\right \Vert \leq Ck^{2j+2},
\end{equation}
 where $C$ is a constant depending only on $j$, $T$, $\mu$, $M$, a Lipschitz constant on $f$ and $u$ as element of $C^{2j+3}\left( [0,T],H^2(\Omega)\right) $. Substituting (\ref{d8b}) in (\ref{d7}), we have
$$
\|\Theta^{2j+2,n+1}\|^2-\|\Theta^{2j+2,n}\|^2  +2k\gamma \|\nabla \widehat{\Theta}^{2j+2,n+1}\|^2  \leq Ck^{4j+5}, $$and it follows by induction, taking (\ref{d8}) into account, that
$$
\|\Theta^{2j+2,n+1}\|^2+2k\gamma \sum_{i=j}^n\|\nabla \widehat{\Theta}^{2j+2,i}\|^2  \leq Ck^{4j+4}. $$
Inequality  (\ref{d3}) follows from (\ref{d8c}) and the last inequality. 
\end{proof}

The following lemma proves DCC for the implicit midpoint scheme (\ref{c1}) and, together with Theorem \ref{theo:3}, constitutes the main argument for the proof of the general convergence theorem for the DC schemes (\ref{c2}). 
\begin{lemma}
\label{lem:2}
The sequence $\left\lbrace  \bar{u}^{2,n}\right\rbrace _n$ from the scheme (\ref{c1}) approximates $\bar{u}$, exact solution of (\ref{aa1}), with order 2 of accuracy.  Furthermore, if  $u(t)=u_0$ for all $t\in [0,(2p+1)k_0]$, where $k_0$ is the initial time step defined in the introduction ($k_0\mu_0<2$), we have
\begin{equation}
\begin{aligned}
\label{d9}
\|D_-(D_+D_-)^{m}\Theta^{2,n+1} \|+\left \Vert (D_+D_-)^{m}\right. &\left.  {\Theta }^{2,n+1} \right \Vert
+\left \Vert (D_+D_-)^{m} \widehat{\Theta}^{2,n+1} \right \Vert_2\\&+\left(  \gamma k\sum_{i=m}^{n}\|\nabla (D_+D_-)^{m} D\widehat{\Theta}^{2,i+1/2}\|^2\right)^{1/2} \leq Ck^2,
\end{aligned}
\end{equation}
for $m=0,1,2,...,p$, $n=m,m+1,\cdots,N-m$, and $k\leq k_0$, where  $\Theta^{2,n}=\bar{u}^{2,n}-\bar{u}(t_n)$, for $n=0,1,2,\cdots,N$, $\mu_0$ is from inequality (\ref{a3}), and $C$ is a constant depending only on $T$, $\Omega$, $\mu_0$, $k_0$, $M$, the continuity of the source term $S$, the derivatives of $f$ up to order $2m+2$, and the derivatives of $u$ with respect to the time variable $t$ up to order $2m+4$. The parameter $m$ is needed to expect order $2p+2$ of accuracy in $p$ corrections.
\end{lemma}

\begin{proof} Identity (\ref{d4}) for $j=0$ and $\Gamma^0E\bar{u}^{0,n+1/2}=0=\left(\Gamma^0-\Lambda^0 \right)D\left( \bar{u}^{0,n+1/2}-\bar{u}(t_{n+1/2})\right)$ implies that the sequence $\left\lbrace   \bar{u}^{2,n} \right\rbrace _n$ from the scheme (\ref{c1}) approximates $\bar{u}$ with order 2 of accuracy in time, and
\begin{equation}
\label{d10}
\|\Theta^{2,n}\|^2+\gamma k\sum_{i=0}^n\|\nabla \widehat{\Theta}^{2,i}\|^2\leq Ck^{4}, \mbox{ for }n=0,1,\cdots,N,
\end{equation}
where $C$ is a constant depending only on $T$, $\Omega$, a Lipschitz constant on $f$ and the derivatives of $u\in C^3\left( [0,T],H^2(\Omega) \right) $. To prove (\ref{d9}) we proceed by induction on the integer $m$.
\vspace*{.25cm}

\noindent
1)~~The case $m=0$. 

Combining (\ref{aa1}) and (\ref{c1}), we obtain the identity
\begin{equation}
\label{d11}
D\Theta^{2,n+1/2}-M\Delta \widehat{\Theta}^{2,n+1}+h(t_{n+1})=w^{2,n+1/2},
\end{equation}where
$$h(t_{n+1})=f\left(  E\bar{u}^{2,n+1/2}+\varphi(t_{n+1/2})\right) -f\left( E\bar{u}(t_{n+1/2})+\varphi(t_{n+1/2}) \right) =\int_0^1df\left(K^{n+1}_1\right)(\widehat{\Theta}^{2,n+1})d\tau_1,$$
with
$$K_1^{n+1}=E\bar{u}(t_{n+1/2})+\varphi(t_{n+1/2})+\tau_1\widehat{\Theta}^{2,n},$$and
$$
\begin{aligned}
w^{2,n+1/2}=\left[ \bar{u}'(t_{n+1/2})-D\bar{u}(t_{n+1/2})\right]& -M\Delta \left( \bar{u}(t_{n+1/2})-E\bar{u}(t_{n+1/2}) \right)\\& -\left[ f\left( \bar{u}(t_{n+1/2})+\varphi(t_{n+1/2})\right) -f\left( E\bar{u}(t_{n+1/2}) +\varphi(t_{n+1/2})\right)\right].
\end{aligned}$$
Applying $D_+$ to (\ref{d11}), we obtain 
\begin{equation*}
DD_+\Theta^{2,n+1/2}-M\Delta D_+ \widehat{\Theta}^{2,n+1}+D_+h(t_{n+1})=D_+w^{2,n+1/2},
\end{equation*}
and the inner product of this identity with $D_+\widehat{\Theta}^{2,n+1}$ yields
\begin{equation}
\begin{aligned}
\label{d12}
\|D_+\Theta^{2,n+1}\|^2-\|D_+\Theta^{2,n}\|^2+2\gamma k\|\nabla D_+ \widehat{\Theta}^{2,n+1}\|^2\leq 
2k\left( -D_+h(t_{n+1})+D_+w^{2,n+1/2},D_+\widehat{\Theta}^{2,n+1} \right) .
\end{aligned}
\end{equation}
We can write 
\begin{equation}
\label{d13}
\begin{split}
D_+h(t_{n})=\int_0^1df\left(K^{n+1}_1\right)(D_+\widehat{\Theta}^{2,n})d\tau_1+\int_0^1\int_0^1d^2f\left(K^{n}_2\right)\left( D_+K_1^{n},\widehat{\Theta}^{2,n}\right) d\tau_1d\tau_2,
\end{split}
\end{equation}where, for $n+i\leq N$, we have
\begin{equation}
\label{d14}
K_{i+1}^{n}=K^n_{i}+\tau_{i+1}(K_{i}^{n+1}-K_{i}^{n})=K_1^n+\sum_{l=1}^i \sum_{2\leq i_1<\cdots <i_l \leq i+1}\tau_{i_1}\cdots\tau_{i_l}k^lD_+^lK_1^n.
\end{equation}
The scheme (\ref{c1}) can be transformed into equations (\ref{c3})-(\ref{c4}), substituting $k$ by $k/2$ and choosing $u=E\bar{u}^{2,n+1}$, $\sigma=\varphi( t_{n+1/2})$ and $v=(k/2)\tilde{S}(t_{n+1/2})+\bar{u}^{2,n}$. It follows from (\ref{c5}), inequality (\ref{d10}) and the triangle inequality that 
\begin{equation*}
\begin{aligned}
 \|E\bar{u}^{2,n+1/2}\|_2\leq C \left(  \|S(t_{n-\frac{1}{2}})\|\right.  &\left. +\|D_-\Theta^{2,n}\|+\| \widehat{\Theta}^{2,n+1}\|_1+\|D_-\bar u(t_{n})\|+\| E\bar {u}(t_{n+1/2})\|_1+\|\varphi(t_{n+1/2})\|_2\right. \\&\left. +\|f( \varphi(t_{n+1/2})) \|_1 \right)^\nu \leq C ,\end{aligned}
\end{equation*}where $C$ is a constant depending only on $\Omega$, the matrix $M$ and $\mu_0$. From inequalities (\ref{b3}), (\ref{d10}) and the Sobolev embedding $H^2(\Omega) \hookrightarrow C^0(\overline{\Omega})$, the last inequality implies the existence of a real $R>0$, depending only on $T$, $\Omega$, the regularity of $S$, the first derivative of $f$, and the second derivative of $u$ with respect to $t$, such that
\begin{equation}
\label{d15} \|K^n_i\|_\infty \leq R, \mbox{ for } i=1,2,\cdots,2p+1.
\end{equation}
From the condition (\ref{a3}) we have 
\begin{equation}
\label{d16}
\left(df\left(K^{n}_1\right)(D_+\widehat{\Theta}^{2,n}),D_+\widehat{\Theta}^{2,n}  \right)\geq -\mu_0 \|D_+\widehat{\Theta}^{2,n}\|^2.
\end{equation}
From (\ref{d15}) and (\ref{b3}) we have, for almost every $x\in \Omega$,
$$
\begin{aligned}
\displaystyle \left \vert d^2f\left(K^{n}_2\right)\left( D_+K_1^{n},\widehat{\Theta}^{2,n+1}\right)(x) \right  \vert & \leq \max_{|y|\leq R} \left \vert d^2f(y)\right \vert |D_+K^n_1(x)||\widehat{\Theta}^{2,n+1}(x)|\\
& \leq C\left(|\widehat{\Theta}^{2,n+1}(x)|+|D_+\widehat{\Theta}^{2,n+1}(x)||\widehat{\Theta}^{2,n+1}(x)| \right). 
\end{aligned}
$$
Therefore,
$$\displaystyle \left \Vert d^2f\left(K^{n}_2\right)\left( D_+K_1^{n},\widehat{\Theta}^{2,n+1}\right) \right  \Vert \leq C \left(\|\widehat{\Theta}^{2,n+1}\|+\| D_+\widehat{\Theta}^{2,n+1}\|_{L^4(\Omega)}\|\widehat{\Theta}^{2,n+1}\|_{L^4(\Omega)}  \right),$$
and we deduce from the Sobolev embedding $H^1(\Omega) \hookrightarrow L^4(\Omega)$ that 
\begin{equation}
\label{d16b}
\left \Vert d^2f\left(K^{n}_2\right)\left( D_+K_1^{n},\widehat{\Theta}^{2,n+1}\right) \right  \Vert \leq C \left(\|\widehat{\Theta}^{2,n+1}\|+\|D_+\widehat{\Theta}^{2,n+1}\|_1\| \widehat{\Theta}^{2,n+1}\|_1  \right).
\end{equation}
This inequality and (\ref{d16}) together with the Cauchy-Schwartz inequality yield
\begin{equation}
\label{d17}
\begin{aligned}
-k( D_+h(t_{n+1}),D_+\widehat{\Theta}^{2,n+1} ) &\leq k\mu_0\|D_+\widehat{\Theta}^{2,n+1}\|^2 +Ck\|D_+\widehat{\Theta}^{2,n+1}\|\left(\|\widehat{\Theta}^{2,n+1}\|+ \|D_+\widehat{\Theta}^{2,n+1}\|_1\|\widehat{\Theta}^{2,n+1}\|_1\right)\\&\leq k\mu_0\|D_+\widehat{\Theta}^{2,n+1}\|^2 +\frac{1}{2}\gamma k\|\nabla D_+\widehat{\Theta}^{2,n+1}\|^2 \\&\quad+Ck\|D_+\widehat{\Theta}^{2,n+1}\|\left[\|\widehat{\Theta}^{2,n+1}\|+ \|D_+\widehat{\Theta}^{2,n+1}\|\left(  \|\widehat{\Theta}^{2,n+1}\|_1+ \|\widehat{\Theta}^{2,n+1}\|_1^2\right) \right]
\end{aligned}
\end{equation}where we have used the Cauchy inequality with $\varepsilon =\gamma /2$:
$$
\begin{aligned}
\|D_+\widehat{\Theta}^{2,n+1}\|\|\nabla D_+\widehat{\Theta}^{2,n+1}\|\| \widehat{\Theta}^{2,n+1}\|_1\leq \frac{\gamma}{2}\|\nabla D_+\widehat{\Theta}^{2,n+1}\|^2+\frac{1}{2\gamma}\|D_+\widehat{\Theta}^{2,n+1}\|^2\|\widehat{\Theta}^{2,n+1}\|^2_1.
\end{aligned}$$
According to (\ref{d10}), we have 
\begin{equation}
\label{d17a}
\|D_+\widehat{\Theta}^{2,n+1}\|\left(  \|\widehat{\Theta}^{2,n+1}\|_1+ \|\widehat{\Theta}^{2,n+1}\|_1^2\right)\leq k^{-1}\left(\| \widehat{\Theta}^{2,n+2}\|+\|\widehat{\Theta}^{2,n+1}\| \right)\left(k^2+k^{3/2}+k^4+k^3 \right) \leq Ck^{5/2}.
\end{equation}
From Taylor's formula with integral remainder we can write 
$$w^{2,n+1/2}=k^2g(t_{n+1}),$$
where, according to  (\ref{b3}), we have
\begin{equation}
\label{d17b}\Vert  D_+^{m_1}D_-^{m_2} g(t_{n})\Vert \leq C, ~\mbox{ for }~ m_2\leq n\leq N-m_1,
\end{equation}
for each nonnegative integers $m_1$ and $m_2$ such that $m_1+m_2\leq 2p+1$. $C$ is a constant depending only on $T$, the derivatives of $f$ up to order $m_1+m_2+1$, and the norm of $u$ in  $C^{m_1+m_2+3}\left([0,T],H^2(\Omega)\right)$. It follows from Cauchy-Schwartz inequality that 
$$\left \vert\left( kD_+w^{2,n+1/2},D_+\widehat{\Theta}^{2,n+1} \right)\right \vert \leq Ck^3\|D_+\widehat{\Theta}^{2,n+1}\|. $$
Substituting the last inequality and the inequality (\ref{d17}) in (\ref{d12}), taking (\ref{d10}) and (\ref{d17a}) into account,  we deduce that 
\begin{equation}
\label{d18}
\begin{aligned}
\|D_+\Theta^{2,n+1}\|^2-&\|D_+\Theta^{2,n}\|^2+\gamma k\|\nabla D_+ \widehat{\Theta}^{2,n+1}\|^2&\leq 2k\mu_0\|D_+\widehat{\Theta}^{2,n+1}\|^2+Ck^3\|D_+\widehat{\Theta}^{2,n+1}\|,
\end{aligned}
\end{equation}
where $C$ is a constant depending only on $T$, $\Omega$, $S$, the second derivative of $f$ and $u\in C^4([0,T],H^2(\Omega))$. This inequality yields 
\begin{equation*}
\|D_+\Theta^{2,n+1}\|-\|D_+\Theta^{2,n}\|\leq k\mu_0 \|D_+\widehat{\Theta}^{2,n+1}\| +Ck^3.
\end{equation*}
 Since $k\mu_0 \leq k_0\mu_0 <2$, it follows by induction that
 $$\|D_+\Theta^{2,n}\|\leq Ck^2\left( \frac{2+k\mu_0}{2-k\mu_0}\right)^n+\left(  \frac{2+k\mu_0}{2-k\mu_0}\right)^n\|D_+\Theta^{2,1}\|.$$
The condition $u(t_n)=u_0$, for $0\leq t_n\leq (2p+1)k_0$, implies $\|D_+\Theta^{2,1}\|=0$. Whence 
\begin{equation}
\label{d18b}\|D_-\Theta^{2,n}\|=\|D_+\Theta^{2,n-1}\|\leq Ck^2, \mbox{ for } n=1,2,\cdots, N.
\end{equation}
Substituting (\ref{d18b}) in the right hand side of (\ref{d18}), we deduce that 
\begin{equation}
\label{d18c}
\|D_-\Theta^{2,n}\|^2+\gamma k\sum_{l=0}^n\|\nabla D_+\widehat{\Theta}^{2,l}\|^2\leq Ck^4.
\end{equation}
On the other hand, by the elliptic regularity results applied to (\ref{d11}), we deduce from (\ref{d15}), (\ref{d17b}) for $m_1=m_2=0$, and (\ref{d18b}) that
$$\|\widehat{\Theta}^{2,n+1}\|_2\leq C\left(\|D_-\Theta^{2,n+1}\|+\|h(t_{n+1})\|+\|w^{2,n+1/2}\| \right)\leq Ck^2.$$
Inequality (\ref{d9}) for $m=0$ holds from (\ref{d10}), (\ref{d18c}) and the last inequality. 

\vspace*{.7cm}
\noindent
2)~~Inequality (\ref{d9}) for $m+1$, assuming that it holds for arbitrary $m\leq p-1$.

We apply $(D_+D_-)^{m+1}$ to the identity (\ref{d11}) and take the inner product of the resulting identity with $(D_+D_-)^{m+1}\widehat{\Theta}^{2,n+1}$ to obtain, as in (\ref{d12}),
\begin{equation}
\begin{aligned}
\label{d19}
\|(D_+D_-)^{m+1}\Theta^{2,n+1}\|^2&-\|(D_+D_-)^{m+1}\Theta^{2,n}\|^2+2\gamma k\|\nabla (D_+D_-)^{m+1} \widehat{\Theta}^{2,n+1}\|^2\\&\leq 
2k\left( -(D_+D_-)^{m+1}h(t_{n+1})+(D_+D_-)^{m+1}w^{2,n+1/2},(D_+D_-)^{m+1}\widehat{\Theta}^{2,n+1} \right) .
\end{aligned}
\end{equation}
As in \cite{koyaguerebo2021arbitrary} we can write
\begin{equation}
\label{d20}
\begin{split}
D_+^sh(t_n)=\sum_{i=1}^{s+1}\sum_{|\alpha_i|=s} L^{n,s}_{i,\alpha_i} ,\mbox{ for } s=1,2,...,2p+1, \mbox{ and }n\leq N-s,
\end{split}
\end{equation}
where $\alpha_i=(\alpha_i^1,\cdots,\alpha_i^{i-1},\alpha_i^i) \in \left\lbrace 1,2,\cdots, s\right\rbrace ^{i-1}\times \left\lbrace 0,1,\cdots, s-i+1\right\rbrace $. $L^{n,s}_{i,\alpha_i}$ is a linear combination, with properly chosen coefficients, of the quantities
\begin{equation*}
L^{n,s}_{i,\alpha_i,\beta_i} =\int_{[0,1]^i}d^iF(K_i^{n+s+1-i})\left( D_+^{\alpha_i^{i-1}}K_{i-1}^{n+\beta_{i}^{i-1}},\cdots,D_+^{\alpha^1_{i}}K_{1}^{n+\beta_i^1},D_+^{\alpha_i^i}\widehat{\Theta}^{2,n+\beta_i^i}\right)d\tau^i,
\end{equation*}
where $\beta_i=(\beta_i^1,\cdots,\beta_i^{i-1},\beta_i^i) \in \left\lbrace 1,2,\cdots, s\right\rbrace ^{i-1}\times \left\lbrace 0,1,\cdots, s-i+1\right\rbrace $ with $\beta_i^l+\alpha_i^{l}\leq s-l+1$, for $l=1, \cdots, i$, and $d\tau^i=d\tau_1\cdots d\tau_i$.
%\begin{equation}
%\label{d20}
%D_+^{2i+2}h(t_n)=\sum_{l=1}^{2i+3}\sum_{|\alpha_l|=2i+2}a_{\alpha_l}L^{n,2i+2}_{l,\alpha_l} ,
%\end{equation}
%where $\alpha_l=(\alpha_l^1,\cdots,\alpha_l^{l-1},\alpha_l^l) \in \left\lbrace 1,2,\cdots,2i+2 \right\rbrace ^{l-1}\times  \left\lbrace 0,1,\cdots,2i+3-l\right\rbrace $,
%\begin{equation*}
%L^{n,2i+2}_{l,\alpha_l} =\int_{[0,1]^l}d^lf(K_l^{n+2i+3-l})\left( D_+^{\alpha_l^{l-1}}K_{l-1}^{n+\beta_{l}^{l-1}},\cdots,D_+^{\alpha^1_{l}}K_{1}^{n+\beta_l^1},D_+^{\alpha_l^l}\widehat{\Theta}^{2,n+\beta_l^l}\right)d\tau,
%\end{equation*}
%$d\tau=d\tau_1 \cdots d\tau_l$, $a_{\alpha_l}$ is a constant and $\beta_l^q$ is a non-negative integer such that $\beta_l^q+\alpha_l^{q}\leq 2i+3-q$, for $q=1,\cdots, l$. 
From (\ref{d15}) and the regularity of $f$ we have
\begin{equation}
\label{d21}\left \Vert d^i f\left( K_{i}^{n} \right) \right \Vert_{\infty} \leq C_i,~~\mbox{ for } i=1,2,...,2p+1,0\leq n\leq N-i+1,
\end{equation}
where $C_i$ is a constant depending only on $T$, the i-th derivative of $f$  and the second derivative of $u$. From the induction  hypothesis (\ref{d9}), the Sobolev embedding $H^2(\Omega)\hookrightarrow L^\infty(\Omega)$, and inequality (\ref{b3}), we have
\begin{equation}
\label{d22}\Vert D^l_+K_{i}^{n}\Vert_\infty \leq C, \mbox{ for } 1\leq l\leq 2m+2,0\leq n\leq N-i-l+1,
\end{equation}
and 
\begin{equation}
\label{d23}
\Vert D_+^l \widehat{\Theta}^{2,n}\Vert \leq Ck^2, \mbox{ for } 1\leq l\leq 2m+1,0\leq n\leq N-l .
\end{equation}

\noindent
-\,\, For $i=1$ we have 
\begin{equation*}
L^{n,s}_{1,\alpha_1} =\int_0^1df(K_1^{n+s})\left( D_+^{s}\widehat{\Theta}^{2,n}\right)d\tau,
\end{equation*}
and, by taking $s=2m+2$, it follows from (\ref{a3}) that 
\begin{equation}
\label{d24} \left( L^{n-m,2m+2}_{1,\alpha_1}, (D_+D_-)^{m+1}\widehat{\Theta}^{2,n+1}\right)\geq -\mu_0\| (D_+D_-)^{m+1}\widehat{\Theta}^{2,n+1}\|^2
\end{equation}
since
$$ D_+^{2m+2}\widehat{\Theta}^{2,n-m}=(D_+D_-)^{m+1}\widehat{\Theta}^{2,n+1}.$$

\noindent
-\,\, For $i=2$ and $|\alpha_2|\leq 2m+2$, we have $1\leq \alpha_2^1\leq 2m+2$ and $0\leq \alpha_2^2\leq 2m+1$. 
%\begin{equation*}
%L^{n,s}_{2,\alpha_2,\beta_2} =\int_{[0,1]\times [0,1]}d^2f(K_2^{n+s-1})\left( D_+^{\alpha_2^{1}}K_{1}^{n+\beta_{2}^{1}},D_+^{\alpha_2^2}\widehat{\Theta}^{2,n+\beta_2^2}\right)d\tau_1\tau_2, ~|\alpha_2|=2i+2.
%\end{equation*}
It follows by the triangle inequality, the inequalities (\ref{b3}) and (\ref{d21})-(\ref{d23}) that
 \begin{equation}
 \label{d25}\|L^{n,s^*}_{2,\alpha_2,\beta_2}\|\leq \left \Vert d^2 f\left( K_{2}^{n+s^*-1} \right) \right \Vert_{\infty}  \|D_+^{\alpha_2^{1}}K_{1}^{n+\beta_{2}^{1}}\|_\infty \Vert  D_+^{\alpha_2^2}\widehat{\Theta}^{2,n}\Vert \leq Ck^2, \mbox{ for } s^*\leq 2m+2.
 \end{equation}

\noindent 
-\,\, For $i\geq 3$ and $|\alpha_i|\leq 2m+3$, we have $1\leq \alpha_i^l \leq 2m+2$, for $l=1,2,\cdots, i-1$, and $0\leq \alpha_i^i \leq 2m+1$. It follows by the triangle inequality,  the inequalities (\ref{b3}) and (\ref{d21})-(\ref{d23}) that, for $s^*\leq 2m+3$,
\begin{equation}
 \label{d26}
 \begin{aligned}
\|L^{n,s^*}_{i,\alpha_i,\beta_i}\|\leq \|d^if(K_i^{n+s^*+1-i})\|_{\infty} \|D_+^{\alpha_i^i}\widehat{\Theta}^{2,n+\beta_i^i}\|\prod_{l=1}^{i-1}\| D_+^{\alpha_i^{l}}K_{l}^{n+\beta_{i}^{l}}\|_{\infty}\leq Ck^2.
\end{aligned}
 \end{equation}
From the identity (\ref{d20}), inequalities (\ref{d24})-(\ref{d26}) yield
\begin{equation}
\label{d27}
\begin{aligned}
\left( -(D_+D_-)^{m+1}h(t_{n+1}),(D_+D_-)^{m+1}\widehat{\Theta}^{2,n+1}\right) &\leq \mu_0\|(D_+D_-)^{m+1} \widehat{\Theta}^{2,n+1}\|^2+ Ck^2\|(D_+D_-)^{m+1}\widehat{\Theta}^{2,n+1}\|.
\end{aligned}
\end{equation}
From inequality (\ref{d17b}) we have
\begin{equation}
\label{d28}
\|(D_+D_-)^{m+1}w^{2,n+1/2} \|\leq Ck^2.
\end{equation}
Substituting (\ref{d27}) and (\ref{d28}) in (\ref{d19}), we obtain
\begin{equation}
\begin{aligned}
\label{d29}
\|(D_+D_-)^{m+1}\Theta^{2,n+1}\|^2&-\|(D_+D_-)^{m+1}\Theta^{2,n}\|^2+2\gamma k\|\nabla (D_+D_-)^{m+1}\widehat{\Theta}^{2,n+1}\|^2\\&\leq 2k\mu_0 \| (D_+D_-)^{m+1}\widehat{\Theta}^{2,n+1}\|^2+Ck^3\|(D_+D_-)^{m+1}\widehat{\Theta}^{2,n+1}\|.
\end{aligned}
\end{equation}
Proceeding as in (\ref{d18}), we deduce by induction that  
$$\|(D_+D_-)^{m+1}\Theta^{2,n}\|\leq \left( Ck^2+\|(D_+D_-)^{m+1}\Theta^{2,m+1}\|\right)\left( \frac{2+k\mu_0}{2-k\mu_0}\right)^{n-m-1}. $$
Since $u(t_n)=u_0$ for $0\leq t_n\leq (2p+1)k_0$, we have $\|(D_+D_-)^{m+1}\Theta^{2,m+1}\|=0$, for $m\leq p-1$. Whence
\begin{equation}
\label{d29b}
\|(D_+D_-)^{m+1}\Theta^{2,n}\|\leq Ck^2, \mbox{ for } n=m+1,m+2,\cdots, N-m-1.
\end{equation}
Substituting (\ref{d29b}) in the right hand side of (\ref{d29}), we deduce by induction that
$$\|(D_+D_-)^{m+1}\Theta^{2,n} \|^2+2\gamma k\sum_{i=m+1}^{n}\|\nabla (D_+D_-)^{m+1} \widehat{\Theta}^{2,i}\|^2 \leq Ck^4.$$
It is immediate from (\ref{d21})-(\ref{d23}) that
\begin{equation*}
\|L^{n,2m+1}_{1,\alpha_1} \|\leq \|df(K_1^{n+2m+1})\|_{\infty} \|D_+^{2m+1}\widehat{\Theta}^{2,n}\|\leq Ck^2.
\end{equation*}
Therefore, applying $D_-(D_+D_-)^{m}$ to (\ref{d11}), we deduce from the elliptic regularity inequality,  the identity (\ref{d20}), the last inequality, the inequalities (\ref{d25})-(\ref{d26}), (\ref{d29b}) and (\ref{d17b}) that 
$$ \Vert D_-(D_+D_-)^{m} \widehat{\Theta}^{2,n+1}  \Vert_2 \leq \|D_-(D_+D_-)^m \left( D\Theta^{2,n+1/2}+ h(t_{n+1})+w^{2,n+1} \right) \| \leq Ck^2.$$
It follows that
\begin{equation}
\begin{aligned}
\label{d30}
\|(D_+D_-)^{m+1}\Theta^{2,n+1} \|+\left( \gamma k\sum_{i=m+1}^{n}\|\nabla (D_+D_-)^{m+1} \widehat{\Theta}^{2,i}\|^2\right)^{1/2}+\Vert D_-(D_+D_-)^{m} \widehat{\Theta}^{2,n+1} \Vert_2 \leq Ck^2.
\end{aligned}
\end{equation}
Otherwise, applying $D_+(D_+D_-)^{m+1}$ to (\ref{d11}), the same reasoning, taking the induction hypothesis and the inequality (\ref{d30}) into account, yields  (\ref{d9}) for $m+1$. Finally, we deduce by induction that Lemma \ref{lem:2} is true for each $m=0,1,\cdots,  p$. 
\end{proof}

The following theorem shows DCC for the schemes (\ref{c1}) and (\ref{c2}) .
\begin{theorem}
\label{theo:4} Suppose that the exact solution ${u}$ of (\ref{a1}) satisfies ${u}(t)=u_0$ for each $t\in [0,(2p+1)k_0]$,  where  $k_0>0$ is a fixed real such that $k_0\mu_0<2$. Then, for $k\leq k_0$, each sequence $\left\lbrace \bar u^{2j,n}\right\rbrace _n$, $j=1,2,...,p+1$, from the schemes (\ref{c1}) or (\ref{c2})  approximates $\bar{u}$ with order $2j$ of accuracy in time, and we have the estimate
\begin{equation}
\begin{aligned}
\label{d31}
\left \Vert (D_+D_-)^{m} E\left( \bar{u}^{2j,n+1/2}-\bar{u}(t_{n+1/2})\right) \right \Vert_2+\sqrt{k\sum_{i=m}^{n}\|\nabla D_-(D_+D_-)^{m} E\left( \bar{u}^{2j,i+1/2}-\bar{u}(t_{i+1/2})\right)  \|^2}\\+
\|D_-(D_+D_-)^{m}\left( \bar{u}^{2j,n+1}-\bar{u}(t_{n+1})\right) \|+\left \Vert (D_+D_-)^{m} \left( \bar{u}^{2j,n+1}-\bar{u}(t_{n+1})\right)\right \Vert \leq Ck^{2j},
\end{aligned}
\end{equation}
for $m=0,1,...,p-j$ and $n=m+j-1,m+j,..., N-j-m$, where $\mu_0$ is from (\ref{a3}), and $C$ is a constant depending only on $m$, $T$, $\mu_0$, $k_0$, $M$, the function $S$, and the derivatives of  $f$ and $u=u(t)$ up to order $2m+2j$ and $2m+2j+2$, respectively. The parameter $m$ is needed to expect order $2p+2$ of accuracy in $p$ corrections.
\end{theorem}

\begin{proof} We proceed by induction on $j=1,2,...,p+1$, and the case $j=1$ results from Lemma \ref{lem:2}. Suppose that $\left\lbrace \bar u^{2j,n}\right\rbrace _n$ satisfies (\ref{d31}) up to an arbitrary order $j\leq p$. Let us prove that the theorem is still true for $j+1$. 

Since $\left\lbrace \bar{u}^{2j,n}\right\rbrace _n$ satisfies (\ref{d31}), it also satisfies DCC, and then Theorem \ref{theo:3} together with the condition ${u}(t)=u_0$ in $[0,(2p+1)k_0]$ implies that $\left\lbrace \bar{u}^{2j+2,n}\right\rbrace _n$ approximates $\bar{u}$ with order $2j+2$ of accuracy in time. Therefore, it is enough to establish (\ref{d31}) for $j+1$. 
We can rewrite the identity (\ref{d4}) as follows

\begin{equation}
\label{d32}
\begin{split}
D\Theta^{2j+2,n+1/2}-M\Delta \widehat{\Theta}^{2j+2,n+1}+H(t_{n+1})=w^{2j+2,n+1/2},
\end{split}
\end{equation}where
$$H(t_{n+1})=\int_0^1df \left(E\bar{u}(t_{n+1/2})-\Gamma^j E\bar{u}(t_{n+1/2})+\varphi(t_{n+1/2})+\tau_1 \widehat{\Theta}^{2j+2,n+1}\right)\left(\widehat{\Theta}^{2j+2,n+1}  \right) d\tau_1,$$and
$$
w^{2j+2,n+1/2}=\sigma^{2j+2,n+1/2}+(\Lambda^j -\Gamma^j)D\left( \bar{u}^{2j,n+1/2}-\bar{u}(t_{n+1/2})\right) . $$
Here $\Theta^{2j+2,n+1}$ and $\sigma^{2j+2,n+1/2}$ are as in Theorem \ref{theo:3}. From the central finite difference (\ref{b6})-(\ref{b7}) and the regularity of $u$ with respect to $t$, we can write  
\begin{equation*}\sigma^{2j+2,n+1/2} =k^{2j+2}G(t_{n+1/2}),
\end{equation*}where 
\begin{equation*}
\Vert  D_+^{m_1}D_-^{m_2} G(t_{n})\Vert \leq C, ~\mbox{ for }~ m_2\leq n\leq N-m_1,
\end{equation*}
for each nonnegative integers $m_1$ and $m_2$ such that $m_1+m_2\leq 2p-2j+1$. $C$ is a constant depending only on $T$, the derivatives of $f$ up to order $m_1+m_2+2j+1$ and the norm of $u$ in  $C^{m_1+m_2+2j+3}\left([0,T],H^2(\Omega)\right)$. On the other hand, from the induction hypothesis and Remark \ref{rmk:4}, we immediately have
$$\|D_-(D_+D_-)^{m}(\Lambda^j -\Gamma^j)(\bar u^{2j,n}-\bar u(t_{n}))\|\leq Ck^{2j+2}, \mbox{ for } m=0,1,...,p-(j+1).$$
The last two inequalities implies that 
\begin{equation*}
 \Vert  D_+^{m_1}D_-^{m_2} w^{2j+2,n+1/2}\Vert \leq Ck^{2j+2} ,\mbox{ for } m_1+m_2\leq 2p-2j-2,
\end{equation*}
and $ m_2+j\leq n\leq N-m_1-j-1$. Therefore, the reasoning from Lemma \ref{lem:2}, substituting the functions $h$ by $H$, $w^{2,n+1/2}$ by $w^{2j+2,n+1/2}$, $\widehat{\Theta}^{2,n+1}$ by $\widehat{\Theta}^{2j+2,n+1}$ and $k^2$ by $k^{2j+2}$, yields
\begin{equation*}
\begin{aligned}
\|D_-(D_+D_-)^{m}\Theta^{2j+2,n+1} \|+\left \Vert (D_+D_-)^{m} {\Theta }^{2j+2,n+1} \right \Vert
+\left \Vert (D_+D_-)^{m} \widehat{\Theta}^{2j+2,n+1} \right \Vert_2\\+\left(  k\sum_{i=m}^{n}\|\nabla D(D_+D_-)^{m} \widehat{\Theta}^{2j+2,i+1/2}\|^2\right)^{1/2} \leq Ck^{2j+2},
\end{aligned}
\end{equation*}
for $m=0,1,...,p-(j+1)$, and  (\ref{d31}) for $j+1$ follows by the triangle inequality. Inequality (\ref{d31}) then holds for arbitrary integer $j\leq p+1$. 
\end{proof}

\section{Fully discretized schemes and convergence results}
\label{sec:Conv} 
This section is dedicated to the analysis of the full discretization of problem (\ref{aa1}). The fully discrete schemes are constructed as discretizations of the time semi-discrete schemes, given by formulae (\ref{c1}) or (\ref{c2}), via the Galerkin ﬁnite element method. Existence of the fully discrete solutions is proven using an adaptation of the Lemma on zeros of a vector ﬁeld. We prove unconditional convergence of the fully discrete solutions to the exact one. The proof strategy does not use DCC. We simply prove unconditional convergence in space of each fully discrete solution to the corresponding time semi-discrete solution and deduce the global convergence by the triangle inequality. 

 Let $\Gamma$ be the portion of the boundary of $\Omega$ on which Dirichlet boundary conditions are imposed to the exact  solution $u$ of problem (\ref{a1}).  Let $S_h$ be a finite dimensional subspace of $\displaystyle V=\left\lbrace v\in H^1(\Omega)\,|\, v=0 \mbox{ on } \Gamma \right\rbrace $  and $\left\lbrace \phi_i \right\rbrace_{i=1}^{N_h} $ a basis for $S_h$ consisting in continuous piecewise polynomials of degree $r\geq 1$ (see for instance \cite{ern2013theory} for an introduction to finite element subspaces $S_h$; the integer $r$ is related to the regularity of the exact solution of (\ref{a1}) in space). We suppose that there exist an interpolation operator $I^r_h$ from $H^1(\Omega)$ onto $S_h$ and a constant $c>0$ such that $0\leq l\leq r$ implies
\begin{equation}
\label{e1}\|v-I^r_hv\|+h\|\nabla \left( v-I^r_hv \right)  \|\leq ch^{l+1} |v|_{l+1,2,\Omega }, \mbox{  } \forall v\in H^{l+1}(\Omega),
\end{equation}
and
\begin{equation}
\label{e1b}
\|v-I^r_hv\|_{L^4 (\Omega)} +h\|\nabla \left( v-I^r_hv \right)  \|_{L^4 (\Omega)} \leq ch^{l+1} |v|_{l+1,4,\Omega }, \mbox{  } \forall v\in W^{l+1,4}(\Omega),
\end{equation}
where $\vert \cdot \vert_{l+1,\rho,\Omega}$ is the following seminorm in $W^{l+1,\rho}(\Omega)$:
$$\vert v \vert_{l+1,\rho,\Omega}=\sum_{|\alpha|=l+1}\vert \partial^\alpha v\vert_{L^\rho(\Omega)}.$$
We say that $S_h$ satisfies the inverse inequality if
\begin{equation}
\label{e1c}
\|v_h\|_{\infty}\leq ch^{m-d/2}\|v_h\|_{m},\mbox{ ~ } \forall v_h\in S_h,  \mbox{ and } m=0,1.
\end{equation}
The estimates (\ref{e1}) and (\ref{e1b}) hold when $S_h$ is obtained from a shape-regular family of meshes $\left\lbrace  \mathcal{T}_h\right\rbrace_{h>0} $ \cite[Corollary 1.109 \& 1.110 ]{ern2013theory} while (\ref{e1c}) is due to \cite[Theorem 3.2.6]{ciarlet1978finite} or \cite[Lemma 1.142]{ern2013theory} for a family of quasi-uniform meshes. Let $R_h$ be the orthogonal projection of $H^1(\Omega)$ onto $S_h$ with respect to the inner product $(u,v)\mapsto \left( u,v\right) +\left( M\nabla u,\nabla v\right) $ . Proceeding as in \cite[Theorem 1.1]{thomee1984galerkin}, we deduce from (\ref{e1}) that
\begin{equation}
\label{e3}
\Vert R_hv-v\Vert +h\Vert \nabla (R_hv-v)\Vert \leq Ch^{l+1} \|v\|_{H^{l+1}(\Omega)}, \mbox{  } \forall v\in H^{l+1}(\Omega), 0\leq l\leq r.
\end{equation}
Furthermore, if $S_h$ satisfies the inverse inequality (\ref{e1c}), we deduce from (\ref{e3}) and (\ref{e1}) for $l=1$, and (\ref{e1b}) for $l=0$ together with the continuous embedding $H^2(\Omega)\hookrightarrow W^{1,4}(\Omega) \hookrightarrow L^\infty(\Omega)$, that
\begin{equation}
\label{e3b}
\begin{aligned}
\Vert R_h v\Vert_{\infty} &\leq \Vert R_h v-I^r_hv\Vert_{\infty} +\Vert v-I^r_h v\Vert_{\infty} +\Vert v\Vert_{\infty} \leq  ch^{1/2}\Vert v\Vert_2+C\|v\|_2,\,\, \forall v\in H^2(\Omega).
\end{aligned}
\end{equation}

For $j=0,1,2,\cdots, p$ and each positive integer $n\leq N$, we look for a function $ \bar{u}^{2j+2,n}_h \in  H^1(\Omega)$ of the form
\begin{equation}
\label{e4}
\bar{u}^{2j+2,n}_h=\sum_{l=1}^{N_h}U^{2j+2,n}_l\phi_l,
\end{equation}satisfying 
%\begin{equation}
%\label{e5}
%\left( Du^{2,n+1/2}_h,\phi \right)+ \left(\nabla M\widehat{u}^{2,n+1}_h,\nabla \phi \right)+\left( f(\widehat{u}^{2,n+1}_h),\nabla \phi \right)=\left(s(t_{n+1/2}),\phi \right)  
%\end{equation}

\begin{equation}
\label{e5} 
\begin{aligned} &\left( D\bar{u}^{2j+2,n+1/2}_h-\Lambda^j D\bar{u}^{2j,n+1/2}_h,\phi \right)+\left( M\nabla \left(E{\bar{u}}^{2j+2,n+1/2}_h-\Gamma^j E{\bar{u}}^{2j,n+1/2}_h \right),\nabla\phi \right) \\& +\left( f\left(E{\bar{u}}^{2j+2,n+1/2}_h-\Gamma^j E{\bar{u}}^{2j,n+1/2}_h +\varphi(t_{n+1/2})\right),\phi \right)=\left( \tilde{S}(t_{n+1/2}), \phi \right), \forall \phi \in S_h, \mbox{ and } n\geq j
\end{aligned}
\end{equation}
\begin{equation}
\label{e6}
\bar{u}^{2j+2,0}_h=R_h\left( u_0-\varphi(0)\right) ,
\end{equation}
where $\Lambda^j D\bar{u}^{2j,n+1/2}_h=0=\Gamma^j E\bar{u}^{2j,n+1/2}_h$ if $j=0$. The scheme (\ref{e5})-(\ref{e6}), denoted DC(2j+2), constitutes a full discretization of the problem (\ref{aa1}) with deferred correction in time, at the discrete points $0=t_0<t_1<\cdots <t_N=T$, $t_n=nk$, and finite element in space. For the starting values in  (\ref{e5})-(\ref{e6}), $0\leq n\leq j-1$ , we consider the following scheme which is deduced from (\ref{c2b}):
\begin{equation}
\label{e7}
\begin{aligned} &\left( D\bar{u}^{2j+2,n+1/2}_h-\frac{1}{2j+1}\bar{\Lambda}^j D\bar{\bar{u}}^{2j,n_j+1/2}_h+f\left( E\bar{u}^{2j+2,n+1/2}_h-\bar{\Gamma}^j E\bar{ \bar{u}}^{\,2j,n_j+1/2}_h +\varphi(t_{n+1/2})\right),\phi \right)\\&+\left( M\nabla \left(E\bar{u}^{2j+2,n+1/2}_h-\bar{\Gamma}^j E\bar{\bar{u}}^{\,2j,n_j+1/2}_h \right),\nabla\phi \right)=\left( \tilde{S}(t_{n+1/2}), \phi \right), \forall \phi \in S_h, 
\end{aligned}
\end{equation}
\begin{equation}
\label{e8}
\bar{u}^{2j+2,0}_h=R_h\left( u_0-\varphi(0)\right) .
\end{equation}

The following theorem proves the existence of a solution for the schemes (\ref{e5})-(\ref{e6}).

\begin{theorem}[Existence of a solution for the fully discretized scheme] 
\label{theo:5}We suppose that $k\displaystyle \|\tau\circ \varphi\|_\infty <2$. Then, for each $j=1,2,\cdots$, there exists a sequence $\left\lbrace  \bar{u}^{\,2j,n}_h\right\rbrace_{n=0}^N$ of elements of the form (\ref{e4}) satisfying (\ref{e5})-(\ref{e6}).
\end{theorem}

To prove this theorem we need the following lemma which is an adaptation of the lemma on zeros of a vector field \cite[p.531]{evans2010partial}. 

\begin{lemma} 
 \label{lem:3}Let $m$ be a positive integer and $v :\mathbb{R}^m \rightarrow \mathbb{R}^m$ a continuous function satisfying 
\begin{equation}
\label{e9}v\left(z \right) \cdot z \geq 0 ~~\mbox{ if } \|z\|_{\ast}=R,
\end{equation}for a positive real $R$, where $\|.\|_{\ast}$ is an arbitrary norm on $\mathbb{R}^m$. Then there exists a point $z$ in the closed ball
$$\overline{B}(0,R)=\left\lbrace z\in \mathbb{R}^m : \|z\|_{\ast}\leq R \right\rbrace $$such that $v(z)=0.$
\end{lemma}

\begin{proof}[Proof of Lemma \ref{lem:3}] Suppose that $v\left(z\right)\neq 0 $ for each $z\in \overline{B}(0,R)$. The mapping
$$ \psi \,:\,\overline{B}(0,R)\rightarrow \overline{B}(0,R)$$defined by 
$$\psi (z)=-\frac{ R}{\|v\left( z\right)\|_{\ast}}v\left( z \right)$$
is continuous. Since $\overline{B}(0,R)$ is a compact and convex subset of $\mathbb{R}^m$, we deduce from Schauder's fixed-point theorem that $\varphi$ has a fixed point $z \in \overline{B}(0,R)$. Therefore, $\|z\|_\ast=R$, and this leads to the contradiction
$$ 0< |z|^2=\psi(z)\cdot z=-\frac{ R}{\|v\left( z\right)\|_\ast}v\left( z \right)\cdot z\leq 0.$$
\end{proof}

\begin{proof}[Proof of Theorem \ref{theo:5}] We proceed by double induction on $j=1,2,\cdots$ and $n=0,1,\cdots, N$, using Lemma \ref{lem:3} for the function $v~:\mathbb{R}^{N_h} \rightarrow \mathbb{R}^{N_h}$ defined by 
\begin{equation}
\label{e10} v^l(z)=\left( \frac{2z_h-2a_h}{k},\phi_l \right)+\left( M\nabla z_h,\nabla\phi_l \right)+\left( f\left( z_h+\varphi(t_{n+1/2}) \right)- \tilde{S}(t_{n+1/2}), \phi_l\right),
\end{equation}
for $l=1,2\cdots,N_h$, where $a_h\in S_h$ is fixed and $z_h$ is the unique element of $S_h$ associated to $z \in \mathbb{R}^{N_h}$ and defined by
\begin{equation*}
z_h=\sum_{l=1}^{N_h}z_l\phi_l.
\end{equation*}
We take $\|z\|_{\ast}=\|z_h\|$. The function $v$ is continuous. For $j=1$, we have $\bar{u}^{2,0}_h=R_h\left( u_0-\varphi(0)\right)$ and, supposing that $\bar{u}^{2,n}_h$ exists for an arbitrary integer $n<N$ and taking $a_h=\bar{u}^{2,n}_h$ in (\ref{e10}), we have
\begin{equation}
\label{e11}
\begin{aligned}
v\left( z\right)\cdot z& =\left( \frac{2z_h-2u^{2,n}_h}{k},z_h \right)+\left( M\nabla z_h,\nabla z_h \right)+\left( f\left( z_h + \varphi( t_{n+1/2})\right)- \tilde{S}(t_{n+1/2}), z_h\right)\\
&\geq \frac{\|z_h\|}{k}\left[\left( 2+k\tau\left( \varphi( t_{n+1/2})\right) \right)  \|z_h\|-2\|\bar{u}^{2,n}_h\|-k\left( \|f(\varphi( t_{n+1/2}))\|+\|\tilde{S}(t_{n+1/2})\|\right)  \right]\\
&\geq 0,
\end{aligned}
\end{equation}
for 
$$\|z\|_{\ast}=\frac{1}{2+k\tau\left( \varphi( t_{n+1/2})\right)}\left(1+2\|\bar{u}^{2,n}_h\|+k\|\tilde{S} (t_{n+1/2})\|+k\|f(\varphi( t_{n+1/2}))\| \right) :=R.$$
Then, from Lemma \ref{lem:3}, there exists a point $z$ in the closed ball $\overline{B}(0,R)$ of $\left(\mathbb{R}^{N_h}, \|\cdot\|_{\ast} \right) $ such that $v(z)=0$. Taking $$U^{2,n+1}=\left(U^{2,n+1}_1,\cdots,U^{2,n+1}_{N_h}  \right)=2z-U^{2,n},$$
we have
$$v\left( \frac{U^{2,n+1}+U^{2,n}}{2} \right)\cdot e_l=0,$$
for each $e_l$ in the standard basis of $\mathbb{R}^{N_h}$. The last identity implies the existence of $u^{2,n+1}_h$ of the form (\ref{e4}) satisfying (\ref{e5})-(\ref{e6}). Moreover, if $\left\lbrace  \bar{u}^{2j,n}_h\right\rbrace_{n=0}^N$ exists and satisfies (\ref{e5})-(\ref{e6}), for an arbitrary integer $j\geq 1$, then we have $\bar{u}^{2j+2,0}_h=R_h\left( u_0-\varphi(0)\right)$, and the existence of $\bar{u}^{2j+2,n+1}_h$ is immediate from the existence of $\bar{u}_h^{2j+2,n}$, proceeding as in the case $j=1$, taking $a_h=\bar{u}^{2j+2,n}-\Gamma^j E\bar{u}^{2j,n+1/2}_h+0.5k\Lambda^jD\bar{u}^{2j,n+1/2}_h$ in (\ref{e10}).
\end{proof}

 The following theorem shows the convergence and order of accuracy of the fully discretized schemes (\ref{e5})-(\ref{e6}).
 
 \begin{theorem}[Order of convergence of the fully discretized schemes] 
\label{theo:6} Suppose that the exact solution $u$ of (\ref{a1}) is $C^{2p+4}\left([0,T],H^{r+1}(\Omega) \right)$ and satisfies $u(t)=u_0$ for $t\in [0,(2p+1)k_0]$, where $p$ is a positive integer and $k_0>0$ is a real such that $k_0\max\left\lbrace \mu_0,\tau(0)\right\rbrace <2$, $\mu_0$ and $\tau$ are defined in (\ref{a2})-(\ref{a3}). In addition, suppose that $S_h$ satisfies the inverse inequality (\ref{e1c}). Then, for $j=1,2,\cdots,p+1$, the solution $\left\lbrace  \bar{u}^{2j,n}_h\right\rbrace_{n=0}^N$ of the scheme (\ref{e5})-(\ref{e6}) approximates $\bar{u}=u-\varphi$ with order $2j$ of accuracy in time and order $r+1$ in space, that is
\begin{equation}
\label{e13}
\Vert \bar{u}^{2j,n}_h-\bar{u}(t_n) \Vert+h\left \Vert \nabla \left( \bar{u}^{2j,n}_h-\bar{u}(t_n)\right)  \right \Vert \leq C(k^{2j}+h^{r+1 }),
\end{equation}
for $k<k_0$. Furthermore, we have the estimate

\begin{equation}
\label{e14}
\begin{aligned}
\|\bar{u}^{2j,n}_h-&R_h\bar{u}^{2j,n}\|^2_1+k\sum_{i=0}^n\|D( \bar{u}^{2j,i+1/2}_h-R_h\bar{u}^{2j,i+1/2}) \|^2 +2\alpha k\sum_{i=0}^n\|\bar{u}^{2j,i}_h-R_h\bar{u}^{2j,i}\|^q_{L^q(\Omega)}\leq Ch^{2r+2},
\end{aligned}
\end{equation}
where $C$ is a constant depending only on $j$, $T$, $\Omega$, $M$, $k_0$, $\mu_0$ and the derivatives of $S$, $f$ and $u$.
\end{theorem}

\begin{proof} Inequality (\ref{e13}) is immediate from (\ref{e14}) by quadruple triangle  inequality, writing
$$
\begin{aligned}
\bar{u}^{2j,n}_h-\bar{u}(t_n)=\left(\bar{u}^{2j,n}_h-R_h\bar{u}^{2j,n} \right)&-[\bar{u}(t_n)-\bar{u}^{2j,n}]-[ \bar{u}(t_n)-R_h\bar{u}(t_n)]+\left[\bar{u}(t_n)-\bar{u}^{2j,n}-R_h(  \bar{u}(t_n)-\bar{u}^{2j,n}) \right],\end{aligned}$$
and taking (\ref{e3}) and (\ref{d31}) into account. Therefore, we just need to establish (\ref{e14}). We proceed by induction on $j=1,2,\cdots,p+1$. For this purpose, we need the following claim which proof is a straightforward application of the mean value theorem, the triangle inequality, and inequalities (\ref{d31}), (\ref{e3})-(\ref{e3b}).

\begin{claim}
\label{claim:1} There exist $0<k_3\leq k_0$ and $h_1>0$ such that $k\leq k_3$ and $h\leq h_1$ imply, 
\begin{equation}
\label{e01}
\begin{aligned}
\|R_h(E\bar{u}^{2j+2,n+1/2}-\Gamma^j E\bar{u}^{2j,n+1/2})+\varphi(t_{n+1/2})\|_\infty \leq 1+\|{u}\|_{L^\infty(0,T,H^2(\Omega))} ,
\end{aligned}
\end{equation}
and
\begin{equation}
\label{e02}
\|w_h^{2j+2,n+1/2}\| \leq Ch^{r+1},
\end{equation}for each $j=0,1,\cdots,p$, and $n=0,1,\cdots,N$, where we define
\begin{equation}
\label{e03}
\begin{aligned}
&w_h^{2j+2,n+1/2}=+D\left(\bar{u}^{2j+2,n+1/2}-\Lambda^j \bar{u}^{2j,n+1/2}\right)-R_hD\left(\bar{u}^{2j+2,n+1/2}-\Lambda^j \bar{u}^{2j,n+1/2}\right)+\\&f\left(  E\bar{u}^{2j+2,n+1/2}-\Gamma^j E\bar{u}^{2j,n+1/2}+\varphi(t_{n+1/2})\right) -f\left(R_h(E\bar{u}^{2j+2,n+1/2}-\Gamma^j E\bar{u}^{2j,n+1/2})+\varphi(t_{n+1/2})\right),
\end{aligned}
\end{equation}and we set $\bar{u}^{\,0,n}=0$.
\end{claim}

\noindent
1.\quad The case $j=1$. We proceed in two steps:

\vspace*{.5cm}
\noindent
(i)\quad First, we are going to prove the inequality
\begin{equation}
\label{e20b}
\begin{aligned}
\|\bar{u}^{2,n}_h-R_h\bar{u}^{2,n}\|^2+&2\gamma k\sum_{i=0}^n\|\nabla E(\bar{u}^{2,i+1/2}_h-R_h\bar{u}^{2,i+1/2})\|^2+ 2\alpha k\sum_{i=0}^n\|E(\bar{u}^{2,i+1/2}_h-R_h\bar{u}^{2,i+1/2})\|^q_{L^q(\Omega)}\\& \leq C h^{2r+2} .
\end{aligned}
\end{equation}
The scheme (\ref{c1}) yields
$$\left( D \bar u^{2,n+1/2},\phi\right)+\left(M\nabla E\bar {u}^{2,n+1/2},\nabla \phi \right)+\int_{\Omega}f\left( E\bar{u}^{2,n+1}+\varphi(t_{n+1/2})\right) \phi dx=\left(\tilde{S}(t_{n+1/2}),\phi \right), \mbox{ ~ } \forall \phi \in S_h.$$
Therefore, combining this identity and (\ref{e5}), for $j=0$, we deduce that 
\begin{equation}
\label{e15}
\begin{aligned}
\left( D \Theta^{2,n+1/2}_h,\phi\right)+&\left(M\nabla \widehat{\Theta}_h^{2,n+1},\nabla \phi \right)+\int_{\Omega}  \left[ f(E\bar{u}^{2,n+1/2}+\varphi(t_{n+1/2}))-f(R_hE\bar{u}^{2,n+1/2}+\varphi(t_{n+1/2}))\right]\phi dx \\&=\left(w_h^{2,n+1/2},\phi \right)+\left(M\nabla \left( E\bar{u}^{2,n+1/2}-R_hE\bar{u}^{2,n+1/2}\right) ,\nabla \phi \right), \mbox{ ~ } \forall \phi \in S_h,
\end{aligned}
\end{equation}
where 
$$\Theta^{2,n}_h=\bar{u}^{2,n}_h-R_h\bar{u}^{2,n},$$
and $w_h^{2,n+1/2}$ is defined in (\ref{e03}). Hypothesis (\ref{a2}) and inequality (\ref{e01}) yield 
\begin{equation}
\label{e17c}
\int_{\Omega}  \left[ f\left( E\bar{u}_h^{2,n+1/2}+\varphi(t_{n+1/2})\right) -f\left( R_hE\bar{u}^{2,n+1/2}+\varphi(t_{n+1/2})\right) \right]\widehat{\Theta}_h^{2,n+1} dx\geq \alpha \|\widehat{\Theta}_h^{2,n+1}\|^q_{L^q(\Omega)}-\mu\|\widehat{\Theta}_h^{2,n+1}\|^2,
\end{equation}
where
$$\mu=\max_{|y|\leq 1+\|u\|_{L^\infty \left( 0,T;H^2(\Omega)\right) }}|\tau(y)|.$$
From the properties of orthogonal projection we have the identity  
\begin{equation}
\label{e18b}  \left(M\nabla \left( E\bar{u}^{2,n+1/2}-R_hE\bar{u}^{2,n+1/2}\right) ,\nabla \phi \right)=-\left( E\bar{u}^{2,n+1/2}-R_hE\bar{u}^{2,n+1/2} , \phi \right), \forall \phi \in S_h,
\end{equation}which implies that
$$ \left |\left(M\nabla \left( E\bar{u}^{2,n+1/2}-R_hE\bar{u}^{2,n+1/2}\right) ,\nabla \phi \right)\right |\leq \| E\bar{u}^{2,n+1/2}-R_hE\bar{u}^{2,n+1/2}\|\,\|\phi\| \leq Ch^{r+1}\|\phi\|.$$
Therefore, choosing $\phi =\widehat{\Theta}_h^{2,n+1}$ in (\ref{e15}), we deduce from the last inequality, the Cauchy-Schwartz inequality and the inequalities (\ref{e02}) and (\ref{e17c})  that
\begin{equation}
\label{e19}
\begin{aligned}
\left( D \Theta^{2,n+1/2}_h,\widehat{\Theta}_h^{2,n+1} \right)+&\gamma \|\nabla \widehat{\Theta}_h^{2,n+1}\|^2 + \alpha \|\widehat{\Theta}_h^{2,n+1}\|^q_{L^q(\Omega)} \leq Ch^{r+1}\|\widehat{\Theta}_h^{2,n+1}\|+\mu\|\widehat{\Theta}_h^{2,n+1}\|^2,
\end{aligned}
\end{equation}for $0<k\leq k_3$ and $0<h\leq h_1$. This inequality yields
\begin{equation*}
\left( D \Theta^{2,n+1/2}_h,\widehat{\Theta}_h^{2,n+1} \right) \leq Ch^{r+1}\left \|\widehat{\Theta}_h^{2,n+1}\right \|+\mu\left \|\widehat{\Theta}_h^{2,n+1}\right \|^2,
\end{equation*}
and it follows for $0< k\mu \leq k_3\mu<2$ that
$$\left \|\Theta^{2,n+1}_h \right \|\leq C\frac{k}{2-k\mu}h^{r+1}+\frac{2+k\mu}{2-k\mu} \left \|\Theta^{2,n}_h \right \|.$$
Proceeding by induction as in Theorem \ref{theo:3}, the last inequality yields
\begin{equation}
\label{e20}
\begin{aligned}
\left \|\Theta^{2,n}_h \right \|&\leq \left( nkCh^{r+1} +\left \|\Theta^{2,0}_h\right  \|  \right) \left( \frac{2+k\mu}{2-k\mu} \right)^n\leq Ch^{r+1}
\end{aligned}
\end{equation}
since $nk\leq T$ and $\Theta^{2,0}_h=0$. Inequality (\ref{e20b}) follows by substituting (\ref{e20}) in (\ref{e19}). 
\newline

\vspace*{.5cm}
\noindent 
(ii)\quad Now we are going to prove the inequality
\begin{equation}
\label{e20e}
k\sum_{i=0}^n \left \|D \Theta^{2,n+1/2}_h\right \|^2+\gamma \|\nabla {\Theta}_h^{2,n+1}\|^2\leq Ch^{2r+2}  .
\end{equation}
We choose $\phi =D{\Theta}_h^{2,n+1/2}$ in (\ref{e15}) and obtain
\begin{equation}
\label{e20c}
\begin{aligned}
\left \|D \Theta^{2,n+1/2}_h\right \|^2+&\int_{\Omega}  \left[ f\left( E\bar{u}_h^{2,n+1/2}+\varphi(t_{n+1/2})\right) -f\left( R_hE\bar{u}^{2,n+1/2}+\varphi(t_{n+1/2})\right) \right]D \Theta^{2,n+1/2}_h dx \\&+\left(M\nabla \widehat{\Theta}_h^{2,n+1},\nabla D \Theta^{2,n+1/2}_h \right)=\left(w_h^{2,n+1/2}, D \Theta^{2,n+1/2}_h \right).
\end{aligned}
\end{equation} 
We can write
$$\begin{aligned}f( E\bar{u}_h^{2,n+1/2}+\varphi(t_{n+1/2}))& -f( R_hE\bar{u}^{2,n+1/2}+\varphi(t_{n+1/2})) \\&=\int_0^1df\left( R_hE\bar{u}^{2,n+1/2}+\varphi(t_{n+1/2})+\xi \widehat{\Theta}^{2,n+1}_h\right)\left( \widehat{\Theta}^{2,n+1}_h\right) d\xi.\end{aligned}$$
From the inverse inequality (\ref{e1c}) and the inequality (\ref{e20}), we have
\begin{equation}
\label{e20cb}\| \widehat{\Theta}^{2,n+1}_h\|_\infty \leq ch^{-3/2}\|\Theta^{2,n}_h\|\leq  C h^{r-1/2},\; r\geq 1.
\end{equation}
This inequality together with (\ref{e01}) implies that there exists $0<h_2\leq h_1$ such that, for $0< h\leq h_2$, we have
$$\|R_h E\bar{u}^{2,n+1/2}+\varphi(t_{n+1/2})+\xi \widehat{\Theta}^{2,n+1}_h\|_\infty \leq 2+\|u\|_{L^\infty\left( 0,T;H^2(\Omega)\right) }.$$
The last identity yields
\begin{equation}
\begin{aligned}
\label{e20d}
\left \|f\left( E\bar{u}_h^{2,n+1/2}+\varphi(t_{n+1/2})\right) -f\left( R_hE\bar{u}^{2,n+1/2}+\varphi(t_{n+1/2})\right) \right \|&\leq \max_{|y|\leq 2+ \|u\|_{L^\infty \left( 0,T;H^2(\Omega)\right) }}|df(y)|\left \|\widehat{\Theta}^{2,n+1}_h\right \|\\&\leq C\left \|\widehat{\Theta}^{2,n+1}_h\right \|.\end{aligned}
\end{equation}
Substituting  (\ref{e20d}) in (\ref{e20c}), we deduce by Cauchy-Schwartz inequality and (\ref{e02}) that
\begin{equation*}
\begin{aligned}
k \|D \Theta^{2,n+1/2}_h \|^2+&\left(M\nabla {\Theta}_h^{2,n+1},\nabla {\Theta}_h^{2,n+1} \right)-\left(M\nabla {\Theta}_h^{2,n},\nabla {\Theta}_h^{2,n} \right)\leq Ckh^{2r+2},
\end{aligned}
\end{equation*}
for $n=0,1,\cdots, N-1$. It follows the inequality
\begin{equation*}
k\sum_{i=0}^n \left \|D \Theta^{2,n+1/2}_h\right \|^2+\left(M\nabla {\Theta}_h^{2,n+1},\nabla {\Theta}_h^{2,n+1} \right)\leq Cnkh^{2r+2}
\end{equation*}
since ${\Theta}_h^{2,0}=0$. The last inequality gives exactly (\ref{e20e}), where $C$ is a constant depending only on $T$, $\Omega$, $k_{i+1}$, $h_i$, $i=1,2$, and the derivatives of $f$ and $u$. 

Estimates (\ref{e20b}) and (\ref{e20e}) gives (\ref{e14}) for $j=1$.

\vspace*{.5cm}
\noindent
2.~~  Here we prove inequality (\ref{e14}) for $j+1$, assuming that it holds up to order $j$, $1\leq j\leq p$.

From the scheme (\ref{c2}) we have
\begin{equation}
\label{e21}\begin{aligned}
\left( D \bar u^{2j+2,n+1/2}\right. &\left. -\Lambda^j D\bar u^{2j,n+1/2},\phi\right)+\left(M\nabla \left(E \bar {u}^{2j+2,n+1/2}-\Gamma^j E\bar{u}^{2j,n+1/2}\right) ,\nabla \phi \right)\\&+\int_{\Omega}f\left( E\bar{u}^{2j+2,n+1/2}-\Gamma^j E\bar{u}^{2j,n+1/2}+\varphi(t_{n+1/2})\right) \phi dx=\left(\tilde{S}(t_{n+1/2}),\phi \right), \mbox{ ~ } \forall \phi \in S_h.
\end{aligned}
\end{equation}
Combining this identity and (\ref{e5}), we deduce that
\begin{equation}
\label{e22}
\begin{aligned}
& \left(f\left( E\bar{u}^{2j+2,n+1/2}_h-\Gamma^j E\bar{u}_h^{2j,n+1/2}+\varphi(t_{n+1/2})\right) -f\left( R_h(E\bar{u}^{2j+2,n+1/2}-\Gamma^j E\bar{u}^{2j,n+1/2})+\varphi(t_{n+1/2})\right),\phi\right)\\&+\left( D \Theta^{2j+2,n+1/2}_h,\phi\right) +(M\nabla \widehat{\Theta}_h^{2j+2,n+1},\nabla \phi ) =\left( w_h^{2j+2,n+1/2}+(\Lambda^j -\Gamma^j)D( \bar u^{2j,n+1/2}_h-R_h\bar u^{2j,n+1/2} ),\phi\right)\\&-\left( (Id-R_h)\left( E\bar{u}^{2j+2,n+1/2}-\Gamma^j E\bar{u}^{2j,n+1/2}\right)  , \phi \right) , 
\end{aligned}
\end{equation}
for any $\phi \in S_h$, where we define
$$\Theta^{2j+2,n}_h=\bar u^{2j+2,n}_h-R_h\bar u^{2j+2,n}-\Gamma^j \left(  \bar u^{2j,n}_h-R_h\bar u^{2j,n}\right) ,$$
and we use the identity
\begin{equation}
\label{e22b}
\left(M\nabla (Id-R_h)( E\bar{u}^{2j+2,n+1/2}-\Gamma^j E\bar{u}^{2j,n+1/2})  ,\nabla \phi \right)=-\left( (Id-R_h)\left( E\bar{u}^{2j+2,n+1/2}-\Gamma^j E\bar{u}^{2j,n+1/2}\right)  , \phi \right), 
\end{equation}
for any $\phi \in S_h$. $Id$ denotes the identity application. As in (\ref{e17c}) we have
\begin{equation*}
\begin{aligned}\left(f( E\bar{u}^{2j+2,n+1/2}_h-\Gamma^j E\bar{u}_h^{2j,n+1/2}+\varphi(t_{n+1/2}))\right. &\left.  -f( R_h(E\bar{u}^{2j+2,n+1/2}-\Gamma^j E\bar{u}^{2j,n+1/2})+\varphi(t_{n+1/2})), \widehat{\Theta}_h^{2j+2,n+1}\right)  \\& \geq \alpha \|\widehat{\Theta}_h^{2j+2,n+1}\|^q_{L^q(\Omega)}-\mu\|\widehat{\Theta}_h^{2j+2,n+1}\|^2. 
\end{aligned}
\end{equation*}
Therefore, choosing $\phi=\widehat{\Theta}_h^{2j+2,n+1} $ in (\ref{e22}), we deduce by the triangle inequality, the last inequality and (\ref{e02}) that
\begin{equation}
\label{e25b}
\begin{aligned}
( D \Theta^{2j+2,n+1/2}_h,\widehat{\Theta}_h^{2j+2,n+1} )&+\gamma \|\nabla \widehat{\Theta}_h^{2j+2,n+1}\|^2 +\alpha \|\widehat{\Theta}_h^{2,n+1}\|^q_{L^q(\Omega)} \leq \mu\|\widehat{\Theta}_h^{2j+2,n+1}\|^2+\\&\left(C h^{r+1}+\|(\Lambda^j -\Gamma^j)D_-( \bar u_h^{2j,n+1}-R_h\bar u^{2j,n+1} ) \|\right)\|\widehat{\Theta}_h^{2j+2,n+1}\|.
\end{aligned}
\end{equation}
This inequality implies that
\begin{equation}
\label{e25c}
\begin{aligned}
\| \Theta^{2j+2,n+1}_h\|-\|\Theta^{2j+2,n}_h\| \leq  k\mu\|\widehat{\Theta}_h^{2j+2,n+1}\|+k\left(C h^{r+1}+\left \|(\Lambda^j -\Gamma^j)D\left(  \bar u_h^{2j,n+1/2}-R_h\bar u^{2j,n+1/2}\right)  \right \| \right),
\end{aligned}
\end{equation}and we deduce, for $k\mu<2$, that
\begin{equation*}
\begin{aligned}
\| \Theta^{2j+2,n+1}_h\| \leq  \frac{k}{2-k\mu}\left(C h^{r+1}+\left \|(\Lambda^j -\Gamma^j)D\left(  \bar u_h^{2j,n+1/2}-R_h\bar u^{2j,n+1/2} \right)  \right \| \right)+\frac{2+k\mu}{2-k\mu}\left \| \Theta^{2j+2,n}_h\right \|.
\end{aligned}
\end{equation*}It follows by induction that, 
\begin{equation}
\label{e25e}
\begin{aligned}
\| \Theta^{2j+2,n+1}_h\|\leq  &C\left(\frac{2+k\mu}{2-k\mu} \right)^{n-j} \left( h^{r+1}+\|\Theta^{2j+2,j}_h\|\right) \\&+k\left(\frac{2+k\mu}{2-k\mu} \right)^{n-j}\sum_{m=j}^n \|(\Lambda^j -\Gamma^j)D(  \bar u_h^{2j,m+1/2}-R_h\bar u^{2j,m+1/2}) \|  ,
\end{aligned}
\end{equation}
for $n\geq j$, and for $0\leq n\leq j-1$ we have
\begin{equation}
\label{e25ec}
\begin{aligned}
\|\bar{ \Theta}^{2j+2,n+1}_h\|\leq  &C\left(\frac{2+k\mu}{2-k\mu} \right)^{n} \left(h^{r+1}+\|\bar{\Theta}^{2j+2,0}_h\| \right)\\&+k\left( \frac{2+k\mu}{2-k\mu}\right) ^{n}\sum_{m=0}^j\left \|(\bar{\Lambda}^j -\bar{\Gamma}^j)D\left(  \bar{\bar u}^{\, 2j,(2j+1)m+j+1/2}_h-R_h\bar{\bar u}^{2j,(2j+1)m+j+1/2} \right) \right \| ,
\end{aligned}
\end{equation}
where we define
$$\bar{ \Theta}^{2j+2,n}_h=\bar{u}^{2j+2,n}_h-R_h{\bar u}^{2j+2,n}-\bar{\Gamma}^j\left( \bar{\bar u}^{\,2j,(2j+1)n+j+1}_h-R_h\bar{\bar u}^{\,2j,(2j+1)n+j+1}\right) .$$
Since $\left\lbrace \bar{u}^{2j,n}_h\right\rbrace_{n=0}^N $ and $\left\lbrace \bar{\bar u}^{\,2j,m}_h\right\rbrace_{m=0}^j $ are obtained  from the same scheme, but for different time steps $k$ and $k_j=k/(2j+1)$, respectively, as for $\left\lbrace {\bar u}^{2j,n}\right\rbrace_{n=0}^N $ and $\left\lbrace \bar{\bar u}^{\,2j,m}\right\rbrace_{m=0}^j $, we deduce from the induction hypothesis and the formulae (\ref{c2c}) and (\ref{c2d}) that 
\begin{equation}
\label{e25ad}
\begin{aligned}
\|\bar{ \Theta}^{2j+2,0}_h\|_1&=\|\bar{\Gamma}^j\left( \bar{\bar u}^{2j,j+1}_h-R_h\bar{\bar u}^{\, 2j,j+1}\right)\|_1\leq C\sum_{m=0}^{2j}\|\bar{\bar u}^{\,2j,m}_h-R_h\bar{\bar u}^{\,2j,m} \|_1\leq Ch^{r+1},
\end{aligned}
\end{equation}and
$$
\begin{aligned}
& k\sum_{m=0}^j\left \|(\bar{\Lambda}^j -\bar{\Gamma}^j)D\left(  \bar{\bar u}^{\,2j,m_j+1/2}_h-R_h\bar{\bar u}^{\,2j,m_j+1/2} \right) \right \|&\leq C\sqrt{ k\sum_{m=0}^{2j^2+3j} \|D(  \bar{\bar u}^{2j,m+1/2}_h-R_h\bar{\bar u}^{2j,m+1/2})\|^2}\\&\leq  Ch^{r+1},
\end{aligned}$$
where $m_j=(2j+1)m+j$. Substituting the last two inequalities in (\ref{e25ec}), we deduce that 
\begin{equation*}
 \|\bar{ \Theta}^{2j+2,n}_h\|\leq  Ch^{r+1}, \mbox{ for }0\leq n\leq j,
\end{equation*}
and it follows by the triangle inequality and the induction hypothesis that
\begin{equation}
\label{e25ed}
 \|\bar{u}^{2j+2,n}_h-R_h{\bar u}^{2j+2,n}\|\leq  Ch^{r+1}, \mbox{ for }0 \leq n\leq j.
\end{equation}
By the triangle inequality and the induction hypothesis, (\ref{e25ed}) in turn yields
\begin{equation*}
 \| \Theta^{2j+2,j}_h\|\leq  Ch^{r+1} ,
\end{equation*}
and we have from (\ref{c2ab}) and (\ref{c2ac})
\begin{equation*}
\begin{aligned}
k\sum_{m=j}^n \|(\Lambda^j -\Gamma^j)D( \bar u_h^{2j,m+1/2}-R_h\bar u^{2j,m+1/2} )\| \leq C\sqrt{nk}\sqrt{ k\sum_{m=0}^{n+j} \|D( \bar u_h^{2j,m+1/2}-R_h\bar u^{2j,m+1/2})\|^2}\leq  Ch^{r+1}.
\end{aligned}
\end{equation*}
The last two inequalities and (\ref{e25ed}) substituted in (\ref{e25e}) yields
\begin{equation}
\label{e25ga}\|\Theta^{2j+2,n}_h\| \leq Ch^{r+1}, \mbox{ for } j\leq n\leq N,
\end{equation}
and it follows from (\ref{e25b}) and (\ref{e25ed}) that
\begin{equation}
\label{e25g}
\begin{aligned}\|\bar u^{2j+2,n}_h-R_h\bar u^{2j+2,n}\|^2&+ 2\alpha k\sum_{i=0}^n\|E\bar{u}^{2j+2,i+1/2}_h-R_hE\bar{u}^{2j+2,i+1/2}\|^q_{L^q(\Omega)} \leq Ch^{2r+2} .
\end{aligned}
\end{equation}
Otherwise, proceeding as in the step 1-(ii) of this proof, we choose $\phi=D{\Theta}_h^{2j+2,n+1/2}$ in (\ref{e22}) and deduce from (\ref{e25ga}) that 
\begin{equation}
\label{e25j}
\begin{aligned}
k\sum_{i=j}^n\left \|D{\Theta}_h^{2j+2,i+1/2}\right \|^2&+\gamma \left \|\nabla{\Theta}_h^{2j+2,n+1} \right \|^2\leq Ch^{2r+2}+\left( M\nabla{\Theta}_h^{2j+2,j},\nabla {\Theta}_h^{2j+2,j}\right), 
\end{aligned}
\end{equation}for $j\leq n\leq N$, and, for $0\leq n\leq j-1$,
\begin{equation}
\label{e25jb}
\begin{aligned}
k\sum_{i=0}^j\left \|D{\bar{\Theta}}_h^{2j+2,i+1/2}\right \|^2&+\gamma \left \|\nabla\bar{{\Theta}}_h^{2j+2,n+1} \right \|^2\leq Ch^{2r+2} 
\end{aligned}
\end{equation}
since, from Cauchy-Schwartz inequality and (\ref{e25ad}), we have
$$\left \vert \left( M\nabla \bar{{\Theta}}_h^{2j+2,0},\nabla \bar{{\Theta}}_h^{2j+2,0}\right)\right \vert \leq \Vert M\Vert_F \|\nabla \bar{{\Theta}}_h^{2j+2,0}\|^2\leq  Ch^{2r+2}. $$
By the triangle inequality and the induction hypothesis, inequality (\ref{e25jb}) for $n=j-1$ yields
$$\left \vert \left( M\nabla{\Theta}_h^{2j+2,j},\nabla {\Theta}_h^{2j+2,j}\right) \right \vert \leq \Vert M\Vert_F \|\nabla {\Theta}_h^{2j+2,j}\|^2 \leq Ch^{2r+2}. $$
Substituting the last identity in (\ref{e25j}), we deduce from (\ref{e25jb}), the induction hypothesis, and the triangle inequality that
\begin{equation}
\label{e25k}
\begin{aligned}
k\sum_{i=0}^n\|D ( E\bar{u}^{2j+2,i+1/2}_h-R_hE\bar{u}^{2j+2,i+1/2}) \|^2&+\gamma  \|\nabla\left( E\bar{u}^{2j+2,n+1/2}_h-R_hE\bar{u}^{2j+2,n+1/2} \right)\|^2\leq Ch^{2r+2},
\end{aligned}
\end{equation}
for $0\leq n\leq N-1$, where $C$ is a constant depending only on $j$, $T$, $\Omega$, $M$, and the derivatives of $f$ and $u$. Inequality (\ref{e14}) for the case $j+1$ follows from (\ref{e25g}) and (\ref{e25k}). Therefore, we can conclude by induction that the Theorem holds for $1\leq j\leq p+1$.
\end{proof}

\begin{corollary}
\label{cor:6} Under the conditions of Theorem \ref{theo:6}, if $S_h$ does not satisfy the inverse inequality, provided that, in addition to conditions (\ref{a2}) and (\ref{a3}), $f$ satisfies the inequality
\begin{equation}
\label{e47}\vert f(x)-f(y)\vert\leq C\left(|x-y|+|x-y|^{q-1}\right) , \mbox{ for each } x,y\in\mathbb{R}^J,
\end{equation}
then the solution $\left\lbrace  \bar u^{2j,n}_h\right\rbrace_{n=0}^N$, $1\leq j\leq p+1$,  of the scheme (\ref{e5})-(\ref{e6}) satisfies
\begin{equation}
\label{e48}
\Vert\bar u^{2j,n}_h-\bar u(t_n) \Vert\leq C(h^{r}+k^{2j}),\quad \forall n=0,1,...,N, k<k_0.
\end{equation}
Furthermore, we have the estimate
\begin{equation}
\label{e49}
\begin{aligned}
\|\bar u^{2j,n}_h-I^r_h\bar u^{2j,n}\|^2+&\gamma k\sum_{i=0}^n\|\nabla E( \bar u^{2j,i+1/2}_h-I^r_h\bar u^{2j,i+1/2}) \|^2+2\alpha k\sum_{i=0}^n\|\bar u^{2j,i}_h-I^r_h\bar u^{2j,i}\|^q_{L^q(\Omega)} \leq Ch^{2r},
\end{aligned}
\end{equation}
where $C$ is a constant depending only on $j$, $T$, $\Omega$, $M$, $k_0$, $\mu_0$, and the derivatives of $S$, $f$ and $u$.
\end{corollary}

\begin{proof} We proceed by induction on $j=1,2,\cdots,p$, as in Theorem \ref{theo:6}, and the case $j=1$ is obvious.  The order of accuracy in space is reduced  since, instead of identities (\ref{e18b}) and (\ref{e22b}), we have
\begin{equation*}
\begin{aligned}
&\left \vert \left(M\nabla (Id-I^r_h)( E\bar{u}^{2j+2,n+1/2}-\Gamma^j E\bar{u}^{2j,n+1/2})  ,\nabla \phi \right) \right \vert \\ &\leq \|M\nabla (Id-I^r_h)( E\bar{u}^{2j+2,n+1/2}-\Gamma^j E\bar{u}^{2j,n+1/2})\|\|\nabla \phi\|\leq Ch^{r}\|\nabla \phi\|,
\end{aligned}
\end{equation*}
for each $\phi \in S_h$. If (\ref{e48})-(\ref{e49}) hold up to order $j$, $1\leq j\leq p$, we obtain the inequality (\ref{e25b}) substituting $R_h$ by $I_h^r$, $r+1$ by $r$, and $\Theta^{2j+2,n}_h$ by 
$$\Phi^{2j+2,n}_h=\bar u^{2j+2,n}_h-I_h\bar u^{2j+2,n}-\Gamma^j \left(  \bar u^{2j,n}_h-I_h\bar u^{2j,n}\right);$$that is
\begin{equation*}
\begin{aligned}
( D \Phi^{2j+2,n+1/2}_h,\widehat{\Phi}_h^{2j+2,n+1} )&+\gamma \|\nabla \widehat{\Phi}_h^{2j+2,n+1}\|^2 +\alpha \|\widehat{\Phi}_h^{2,n+1}\|^q_{L^q(\Omega)} \leq \mu\|\widehat{\Phi}_h^{2j+2,n+1}\|^2+\\&\left(C h^{r+1}+\|(\Lambda^j -\Gamma^j)D_-( \bar u_h^{2j,n+1}-I^r_h\bar u^{2j,n+1} ) \|\right)\|\widehat{\Phi}_h^{2j+2,n+1}\|.
\end{aligned}
\end{equation*}
In fact, Claim \ref{claim:1} for $R_h$ substituted by $I^r_h$ is still true from (\ref{e1}) and (\ref{e1b}). For each $\phi\in S_h$, the quantity 
 $$\left(  \left( \Lambda^j -\Gamma^j\right) D\left(  \bar u_h^{2j,i+1/2}-I^r_h\bar u^{2j,i+1/2} \right), \phi\right) $$ is a linear combination of quantities
 $$\begin{aligned} &\left(\Omega^{m+1}_{s},\phi \right)=
\left( M\nabla \left[E \left(\bar{u}^{2s,m+1/2}_h-\bar{u}^{2s,m+1/2} \right) -\Gamma^{s-2} )( \bar u^{2s-2,m+1/2}_h-\bar u^{2s-2,m+1/2})\right] ,\nabla \phi\right)\\&+\left(f( E\bar{u}^{2s,n+1/2}_h-\Gamma^{s-2} E\bar{u}_h^{2s-2,m+1/2}+\varphi(t_{m+1/2})) -f(E\bar{u}^{2s,m+1/2}-\Gamma^{s-2} E\bar{u}^{s-2,m+1/2}+\varphi(t_{m+1/2})),\phi\right), 
\end{aligned}   $$  
for suitable integer $m$ and $s$, where the coefficients of the linear combination dependent only on $j$ and the coefficients of the finite elements formulae (\ref{b6})-(\ref{02}). It follows from hypothesis (\ref{e47}) and the induction hypothesis, taking into account the scheme (\ref{e7})-(\ref{e8}), that 
$$\sum_{i=0}^n\|(\Lambda^j -\Gamma^j)D_-( \bar u_h^{2j,i+1/2}-I_h\bar u^{2j,i+1/2} ) \|^2\leq  Ch^{2r},$$which completes the proof.
\end{proof}

 \section{Numerical illustration}
\label{sec:numerical experiments}
For the numerical illustration we consider two problems. The first is a bistable reaction-diffusion equation that addresses stiffness. The second problem, taken from \cite{alonso2002runge}, is linear with inhomogeneous boundary conditions and addresses order reduction due to a possible ill-treatment of boundary conditions.
\subsection{Problem 1: Bistable reaction diffusion equation}
\begin{equation}\displaystyle 
\label{f1} \begin{aligned}
u_t-u_{xx}+10^4u(u-1)(u-0.25)&=0 \mbox{ in } \Omega\times (0,T),\\
\frac{\partial u}{\partial n}&=0\mbox{ on } \partial \Omega\times (0,T),\\
u(0)&=e^{-100x^2} \mbox{ in } \Omega.\\
\end{aligned}
\end{equation}
We choose $\Omega =(0,1)$ and $T=0.0295$. We are interested in the order of convergence in time. For this purpose, we simply use $P_1$ Lagrange finite elements in space with uniform mesh and the step $h=10^{-3}$. We compute a reference solution using DC10 with the time step $k=1.64\times 10^{-5}$ (N=1800). Table  \ref{tab:4}  gives the maximal absolute error in time, norm $L^2(\Omega)$ in space, and the order of convergence for each pair of consecutive time steps.

 For this problem, we have $$f(u)=10^4u(u-1)(u-0.25),$$and inequalities (\ref{a2}) and (\ref{a3}) hold with $\tau (0)=-1500$ and $\mu_0=8125/3$.  Therefore, according to Theorem \ref{theo:6}, the maximal time step to solve the problem with the DC methods is $k_0=6/8125\simeq 7.38\times 10^{-4}$, that is $N=39.9479\simeq 40$.

 \begin{table}[!ht]
\caption{Absolute error (order of convergence) in for the bistable reaction-diffusion equation}
\label{tab:4}       % Give a unique label
% For LaTeX tables use
\begin{tabular}{llllllll}
\hline\noalign{\smallskip}
 $N$  &  \centering DC2 & \centering DC4 & \centering DC6 &\centering DC8& DC10\\[.5ex]
 40   &  0.115         & 4.62e-03        & 9.14e-04         & 1.97e-04      & 1.11e-03\\[.5ex]
 90   &  8.48e-04(3.21) & 4.59e-05(5.68) & 2.05e-06(7.52)    & 1.55e-06(5.97) & 1.45e-06(8.22)\\[.5ex]
 180  &  5.91e-05(3.84) & 2.17e-06(7.72) & 5.53e-09(8.53)    & 4.09e-09(8.56)    & 1.90e-09(9.57)\\[.5ex]
 360  &  3.87e-06(3.93) & 8.59e-10(7.98) & 2.57e-12(11.07) &  4.51e-13(13.15) & 8.57e-14(14.44)\\[.5ex]
 450  &  1.55e-06(3.96) & 1.44e-10(8.01) &  2.33e-13(10.74) & 2.40e-14(13.14) &  2.48e-15(15.88)  \\[.5ex]
 900  &  9.97e-08(4.00) & 5.63e-13(7.99) &  2.67e-16(9.77) & 8.62e-19(14.75) & 7.36e-21(18.36)\\[.5ex]
 1800 &  6.25e-09(3.99) & 2.18e-15(8.00) & 2.13e-19(10.29) &  1.74e-22(12.27) & --          \\[.5ex]
% 3600 & 1.9e-5~(2.00) & 2.9e-9~(4.00) & 1.0e-11~(5.5)  & 2.2e-13~(5.9) & 9.9e-15~(5.8)\\[.7ex]
% 28800& 3.1e-7~(2.00) & 7.1e-13~(4.0) & 1.6e-15~(4.2)  & 1.2e-15~(2.5) & --  \\[.7ex]
\hline\noalign{\smallskip}
%Order & \centering 2.01 & \centering 3.99 & 6.03 & 7.94 & 9.90\\[1.2ex]
%\noalign{\smallskip}\hline
\end{tabular}
\end{table}

For the computational effort of the DC methods, we recall that to compute an approximate solution at the discrete points $0=t_0<t_1<\cdots <t_N=T$, $DC2 $ solves $N$ nonlinear systems while $DC2j$, $j\geq 2$, solves $j\times N$ systems. For the bistable reaction-diffusion equation, it is clear that, for $N>180$, higher order DC method have the smallest maximal error by solving less systems of equations. For example, $DC10$ achieves absolute error of about $2.48\times 10^{-15}$ by solving approximately $2250$ nonlinear systems while $DC4$ achieves almost the same  accuracy by solving 3600 nonlinear systems. $DC10$, $DC8$, $DC4$ and $DC2$ solve approximately 1800 nonlinear systems, but the corresponding errors are, respectively, $8.57\times 10^{-14}$, $2.4\times 10^{-14}$,  $5.63\times 10^{-13}$ and $6.25\times 10^{-9}$. Since the resolution of nonlinear systems is the main burden for these methods, using high order DC methods is advantageous.

\subsection{Problem 2 (see \cite{alonso2002runge})}
\begin{equation}\displaystyle 
\label{f2} \begin{aligned}
u_t-u_{xx}&=e^t\left(x^2-2x-1.25 \right),\quad 0\leq x\leq 1,\, 0\leq t\leq 1,\\
u(0)&=x^2-2x+0.75, \quad 0\leq x\leq 1,
\end{aligned}
\end{equation}
with exact solution $u(x,t)=e^t\left(x^2-2x+0.75\right)$. We  transform this problem in the form (\ref{aa1}) by choosing
$$\varphi(x,t)=(1-x)u(0,t)+xu(1,t), \mbox{ for Dirichlet boundary condition (DBC)};$$  
$$\varphi(x,t)=\left(x-\frac{1}{2}x^2\right)\frac{\partial u}{\partial x}(0,t)+\frac{1}{2}x^2\frac{\partial u}{\partial x}(1,t), \mbox{ for Neumann boundary condition (NBC)};$$and, 
$$\varphi(x,t)=(x-1)\left[  \frac{\partial u}{\partial x}(0,t)-u(1,t)\right] +xu(1,t), \mbox{ for mixed boundary condition (MBC)}.$$  
We use $P_1$ Lagrange finite elements in space with uniform mesh and the space step $h= 2.5\times 10^{-4}$ for Dirichlet and mixed Dirichlet-Neumann boundary conditions, and $h=10^{-2}$ for the  Neumann boundary condition. Table  \ref{tab:5}  gives the maximal absolute error in time, with norm $L^2(\Omega)$ in space, and the order of convergence for each pair of consecutive time steps. This table shows that the schemes DC2, DC4 and DC6 achieve their theoretical order on this problem.  DC8 and DC10 are more accurate, but their order of accuracy is not observed since errors for these methods are close/equal to the machine accuracy from the largest time step $k=0.2$ ($N=5$).
%\begin{table}[!ht]
%\caption{Absolute error (order of convergence) for the bistable reaction-diffusion equation}
%\label{tab:5}       % Give a unique label
%% For LaTeX tables use
%\begin{tabular}{llllllll}
%\hline\noalign{\smallskip}
% $N$  &  \centering DC2 & \centering DC4 & \centering DC6 &\centering DC8& DC10\\[.7ex]
%% \hline\noalign{\smallskip}
%%&&&DBC&&\\\hline\noalign{\smallskip}
% 4   &  2.62e-04         & 1.19e-07        & 5.46e-10         & 4.47e-11      & 2.94e-13\\[.7ex]
% 8   &  1.65e-05(3.99) & 1.31e-09(6.49) & 4.34e-12(6.98)    & 8.87e-13(5.74) & 9.25e-13\\[.7ex]
% 16  &  1.03e-06(3.99) & 9.60e-12(7.09) & 2.97e-13(3.87)       & 2.96e-13    & 2.96e-13\\[.7ex]
%%  \hline\noalign{\smallskip}
%%&&&NBC&&\\\hline\noalign{\smallskip}
%%% 360  &  3.87e-06(3.93) & 8.59e-10(7.98) & 2.57e-12(11.07) &  4.51e-13(13.15) & 8.57e-14(14.44)\\[.7ex]
%%% 450  &  1.55e-06(3.96) & 1.44e-10(8.01) &  2.33e-13(10.74) & 2.40e-14(13.14) &  2.48e-15(15.88)  \\[.7ex]
%%% 900  &  9.97e-08(4.00) & 5.63e-13(7.99) &  2.67e-16(9.77) & 8.62e-19(14.75) & 7.36e-21(18.36)\\[.7ex]
%%  \hline\noalign{\smallskip}
%%&&&MBC&&\\\hline\noalign{\smallskip}
%% 1800 &  6.25e-09(3.99) & 2.18e-15(8.00) & 2.13e-19(10.29) &  1.74e-22(12.27) & --          \\[.7ex]
%% 3600 & 1.9e-5~(2.00) & 2.9e-9~(4.00) & 1.0e-11~(5.5)  & 2.2e-13~(5.9) & 9.9e-15~(5.8)\\[.7ex]
%% 28800& 3.1e-7~(2.00) & 7.1e-13~(4.0) & 1.6e-15~(4.2)  & 1.2e-15~(2.5) & --  \\[.7ex]
%\hline\noalign{\smallskip}
%%Order & \centering 2.01 & \centering 3.99 & 6.03 & 7.94 & 9.90\\[1.2ex]
%%\noalign{\smallskip}\hline
%\end{tabular}
%\end{table}

\begin{table}[!ht]
\caption{Absolute error (order of convergence) in for problem (\ref{f2})}
\label{tab:5}       % Give a unique label
% For LaTeX tables use
\begin{tabular}{llllllll}
\hline\noalign{\smallskip}
 $N$  &  \centering DC2 & \centering DC4 & \centering DC6 &\centering DC8& DC10\\[.7ex]
 \hline\noalign{\smallskip}
&& \centering DBC& \centering $h=2.5\times 10^{-4}$&&\\[1ex]
\hline\noalign{\smallskip}
 5   &  3.87e-04         & 1.07e-07        & 4.60e-10         & 3.36e-11      & 3.32e-12\\[.7ex]
 10   &  2.43e-05\,(3.99) & 9.38e-10\,(6.83) & 3.11e-12\,(7.21)   & 6.73e-13(5.56) & 4.51e-13(2.88)\\[.7ex]
20  &  1.52e-06\,(3.99) & 6.40e-12\,(7.19) & 3.81e-14\,(6.35)       & 8.01e-15(6.39)    & 3.99e-15(6.82)\\[.7ex]
  \hline\noalign{\smallskip}
&&\centering NBC&\centering $h=10^{-2}$&&\\ [1ex]\hline\noalign{\smallskip}
5   &  3.31e-04         & 2.01e-09        & 6.59e-14         & 0.      & 0.\\[.7ex]
10  &  2.07e-05\,(3.99) & 9.16e-12\,(7.78) & 0.    & 0. & 0.\\[.7ex]
20  &  1.29e-06\,(3.99) & 3.82e-14\,(7.91) & 0.       & 0.    & 0.\\[.7ex]
  \hline\noalign{\smallskip}
&&\centering MBC&\centering $h=2.5\times10 ^{-4}$&&\\[1ex]\hline\noalign{\smallskip}
5   &  4.32e-03      & 8.39e-08        & 1.45e-10         & 3.09e-11      & 4.07e-12.\\[.7ex]
10  &  2.71e-04\,(3.99) & 3.52e-10\,(7.89) & 1.81e-12\,(6.32)    & 2.83e-13(6.77) & 1.44e-13(4.82)\\[.7ex]
20  &  1.69e-05\,(4.00) & 2.93e-12\,(6.91)     & 1.72e-14\,(6.71)       & 2.98e-15(6.57)    & 2.83e-15(5.67)\\[.7ex]
\hline\noalign{\smallskip}
%Order & \centering 2.01 & \centering 3.99 & 6.03 & 7.94 & 9.90\\[1.2ex]
%\noalign{\smallskip}\hline
\end{tabular}
\end{table}

%\bibliographystyle{...}
%\bibliography{...}

\end{document}